\documentclass{amsart}

\usepackage{lineno}
\usepackage{xspace}
\usepackage[numbered,framed]{mcode}
\usepackage{algorithm}
\usepackage{graphicx}
\usepackage{epstopdf}

\renewcommand{\vec}[1]{\mathbf{#1}}

\newcommand\norm[1]{\left\lVert#1\right\rVert _2 }

\newcommand\Lnorm[1]{\left\lVert#1\right\rVert _ {L^2(\Omega,\mathbb{R}^d)} }

\newcommand\eval[1]{\mathbb{E}\left[ #1  \right]    }

\newcommand{\expint}[1]{\textit{#1}}
\newcommand{\ETDint}[1]{\textit{#1}}

\newcommand{\Mil}{MI0}
\newcommand{\MIL}{\mathit{MI0}}
\newcommand{\Exact}{\text{Exact}}

\newcommand{\secref}[1]{Section \ref{#1}\xspace}
\newcommand{\figref}[1]{Figure \ref{#1}\xspace}
\newcommand{\thref}[1]{Theorem \ref{#1}\xspace}

\newtheorem{definition*}{Definition}
\newtheorem{theorem*}{Theorem}
\newtheorem{proposition}{Proposition}
\newtheorem{lemma}{Lemma}
\newtheorem{assumption}{Assumption}
\newtheorem{corollary}{Corollary}

\newcommand{\SG} [1]{\mathbf{\Phi} _{#1} }
\newcommand{\ZZ}[1] {\mathbf{Z} _{#1} }
\newcommand{\FF}[1] {\mathbf{\varphi}(#1)}
\newcommand{\Dt}{\Delta t}

\begin{document}

\title[A New Class of Exponential Integrators for
 SDEs With
 Multiplicative Noise]{A New Class of Exponential Integrators for
 Stochastic Differential Equations With
 Multiplicative Noise}

%    Information for first author
\author{Utku Erdo\u{g}an}
%    Address of record for the research reported here
\address{ Department of Mathematics, Faculty of Art and Science, U\c{s}ak University, 64200, U\c{s}ak,
Turkey.}
%    Current address
%\curraddr{Department of Mathematics and Statistics,
%Case Western Reserve University, Cleveland, Ohio 43403}
\email{utku.erdogan@usak.edu.tr}
%    \thanks will become a 1st page footnote.
%\thanks{The first author was supported in part by NSF Grant \#000000.}

%    Information for second author
\author{ Gabriel J. Lord}
\address{Department of Mathematics and Maxwell Institute for Mathematical Sciences, Heriot Watt University,
EH14 4AS Edinburgh, U.K.}
\email{g.j.lord@hw.ac.uk}
%\thanks{Support information for the second author.}

%    General info

%\date{January 1, 2001 and, in revised form, June 22, 2001.}

%\dedicatory{This paper is dedicated to our advisors.}


\begin{abstract}
In this paper, we present new types of exponential integrators  for
Stochastic Differential Equations (SDEs) that take advantage of
the exact solution of (generalised) geometric Brownian motion.
We examine both Euler and Milstein versions of the scheme and prove 
strong convergence. For the special case of linear noise we obtain an
improved rate of convergence for the Euler version over standard
integration methods. We investigate the efficiency of the methods
compared with other exponential integrators and show that by
introducing a suitable homotopy parameter these schemes are
competitive not only when the noise is linear but also in the presence
of nonlinear noise terms. 
\end{abstract}

\maketitle
\section{Introduction}
We develop new exponential integrators for the numerical approximation stochastic differential equations (SDEs) of the following form
\begin{equation} \label{eq:EqABg}
    d\vec{u}= \left( A \vec{u}+\bf{F}(\vec{u}) \right) dt + \sum _{i=1} ^m  \left(  B_i \vec{u} +\vec{g}_i(\vec{u}) \right)  dW_i (t), \qquad \vec{u} (0)=\vec{u}_0 \in  \mathbb{R} ^d
\end{equation}
where $W_i(t)$ are iid Brownian Motions, $\vec{F}, \vec{g} _i : \mathbb{R} ^d \rightarrow \mathbb{R} ^d$, and   matrices $A,B_i \in \mathbb{R} ^{dxd }$ satisfy the following zero commutator conditions
 \begin{equation*} \label{eq:commute}
 \left[ A,B_i \right]=0, \quad  \left[ B_j,B_i \right]=0 \qquad \text{for} \quad  i,j=1 \hdots m.
 \end{equation*}
%We examine strong convergence and investigate efficiency of the
%schemes we propose applied to a number of problems by numerical experiments.
In the deterministic setting, exponential integrators have proved to
be  very efficient in the numerical solution of stiff (partial)
differential equations when compared to implicit solvers see, for
example, the review in \cite{hochbruck2010exponential}. The 
derivation and  usage of exponential integrators in the stochastic
setting is still an active research area. Local linearisation methods
were first proposed by \cite{jimenez1999simulation, biscay1996local} for
SDEs with both additive and multiplicative noise. These methods continue
to receive attention, see for example \cite{Mora,JimenezCarbonell}
looking at weak approximation and for example  \cite{CarbonellJimenez}
on general noise terms. Recently \cite{YoshiroBurrage} examined
mean square stability of exponential integrators for semi-linear stiff
SDEs. The method is the same basic one as developed
for the space discretisations of SPDEs. For SPDE's with additive
noise, \cite{gevrey2004} introduced an exponential scheme for 
stochastic PDEs 
%with Gevrey regularity that is superior to standard
%semi-implicit Euler- Maruyama scheme. 
and was improved upon in \cite{JentzenKloedenOrder,kloeden2011exponential},
%kloeden2011exponential} The error analysis for the
%pathwise approximation of a general semi-linear stochastic evolution
%equation with additive noise on a Hilbert space is  given in
%\cite{kloeden2011exponential}.  
Jentzen and co-workers (see for
example \cite{JentzenKloedenOrder,JentzenPW,Jentzen} and
references there in) have further extended these results to include
more general nonlinearities. There has been less work on exponential
integrators with multiplicative noise. Strong convergence of
stochastic exponential integrators for  SDEs obtained from space
discretisation of stochastic partial differential equations (SPDEs) by
finite element method is considered in \cite{lord2012} and recently, a
higher order exponential integrator of Milstein type  has 
been introduced by  Jentzen and  R{\"o}ckner \cite{expmil}.

All the above exponential integrators for SDEs (e.g. arising from the
discretisation of the SPDEs) are based on the semi
group operator $\mathbf{S}_{t,t_0}=\exp((t-t_0)A)$ obtained from the
following linear equation 
\begin{equation*}  \label{eq:ETD_semigroup}
d  \mathbf{S}_{t,t_0} = A    \mathbf{S}_{t,t_0}  dt,  \qquad  \mathbf{S}_{t_0,t_0}  =I_d 
\end{equation*}
where $I_d$ is unit matrix in $ \mathbb{R} ^{d \times d}$. 
For comparison, consider the following two 
standard exponential integrators for \eqref{eq:EqABg} with multiplicative
noise: \ETDint{SETD0}
\begin{equation}
  \label{eq:SETD0}
  \vec{u}_n= e^{\Delta tA} \left(\vec{u}_n+  \vec{F} (\vec{u}_n) \Delta t +\sum _{i=1} ^m \left(  B_i \vec{u}_n+ \vec{g}_i(\vec{u}_n )\right) \Delta W_{i,n} \right) 
\end{equation}
and \ETDint{SETD1}
\begin{equation}
  \label{eq:SETD1}
  \vec{u}_n= e^{\Delta t A} \left(\vec{u}_n+ \sum _{i=1} ^m \left(  B_i \vec{u}_n+ \vec{g}_i(\vec{u}_n )\right) \Delta W_{i,n}\right) +  \FF{\Delta t A}  \vec{F} (\vec{u}_n) \Delta t,
\end{equation}
where 
$$\FF{A}  =A^{-1} \left( \exp(A)-I_d\right).$$
These methods are essentially exact for a linear system of ODEs.
We extend this approach to take advantage of the known solution of
geometric Brownian motion in the numerical approximation. To do this,
consider the linear homogeneous matrix differential equation  
\begin{equation} \label{eq:Homogen}
d  \SG{t, t_0} = A  \SG{t,t_0}   dt + \sum _{i=1} ^m  \ B_i \SG{t,
  t_0}dW_i (t), \qquad \SG{t_0, t_0} =I_d
\end{equation} 
and these new schemes are exact for a class of linear systems of
multiplicative SDEs of this form. 

In the next section our new exponential integrators for multiplicative
noise are derived and the homotopy scheme is also introduced. The main
results of strong convergence analysis for the Euler and Milstein
versions of the scheme are stated in \secref{sec:Mainresults} and
numerical examples are presented to examine the efficiency of the
proposed schemes. For linear noise we obtain a strong rate of
$\mathcal{O}(\Dt)$ convergence for Euler type scheme, improving over
standard methods in this case. \secref{sec:proofs} proves strong convergence
of $\mathcal{O}(\Dt)$ for the Milstein version and finally we conclude.
 
\section{Derivation of  the  methods} 
\label{sec:methods}
Throughout we assume that  $T \in (0,\infty)$ is a
fixed real number and we have    a partition of the time interval
$[0,T]$, $0=t_0<t_1<t_2 \hdots  t_N=T$ with constant step size $\Delta
t=t_{j+1}-t_j$. Let $(\Omega , \mathcal{F},\mathbb{P})$ be  a
probability  space with filtration $({\mathcal{F} _ {t}  }) _{t \in
  [0,T]}$. Then  under suitable assumptions on $\vec{F}$ and
$\vec{g}_i$ it is well known that there exists an $\mathcal{F} _ {t^-}$
adapted stochastic process $u:[0,T] \times \Omega  \rightarrow  \mathbb{R} ^d $
satisfying \eqref{eq:EqABg},
\cite{MR1475218,Oksendal2003stochastic,gabrielbook}.   
The linear homogeneous matrix differential equation \eqref{eq:Homogen}
has the exact solution 
\begin{equation*} \label{eq:semigroup}
\SG{t,t_0}=\exp \left(( A-\frac{1}{2} \sum_{i=1} ^m B_i ^2)(t-t_0) +
  \sum_{i=1}^m B_i (W_i (t)  -W_i (t_0)) \right). 
\end{equation*}
Let $\vec{u} (t)$ be the solution of \eqref{eq:EqABg} and take
$t=t_{n+1}$, $t_0=t_n$. Then, applying the Ito formula to
$\vec{Y}(t)=\SG  {t,t_0} ^{-1} \vec{u}$, we obtain 

\begin{multline}  \label{eq:exact}
\vec{u}(t_{n+1})=\SG{t_{n+1},t_n} \left( \vec{u}(t_n) + \int _{t_n} ^{t_{n+1}} \SG {s,t_n} ^{-1} \tilde{f} \left( \vec{u} (s) \right)   ds \right. \\ + \left . \sum _{i=1} ^m \int _{t_n} ^{t_{n+1}} \SG {s,t_n}   ^{-1} \vec{g}_i(\vec{u}(s))    dW_i(s) \right)
\end{multline}
where
\begin{equation} \label{eq:def_of_functions}
  \tilde{\vec{f}} (.)= \vec{F} (.) -\sum _{i=1} ^m B_i \vec{g}_i(.).
\end{equation}
%or equivalently 
%\begin{multline*}  \label{eq:exact_other}
%u(t_{n+1})=\SG{t_{n+1},t_n} \left( \vec{u}(t_n) + \int _{t_n} ^{t_{n+1}} \SG {s,t_n} ^{-1} \tilde{f} \left( \vec{u} (s) \right)   ds \right. \\ + \left . \sum _{i=1} ^m \int _{t_n} ^{t_{n+1}} \SG {s,t_n}   ^{-1} \vec{g}_i(\vec{u}(s))    dW_i(s) \right)
%\end{multline*}
Different treatment of the integrals in \eqref{eq:exact} leads to
different numerical schemes. We examine Euler and Milstein type
methods here, although clearly higher order methods, such as
Wagner-Platen type schemes (see for example
\cite{BeckerJentzenKloeden}) could be developed. 
\subsection{Euler Type Exponential Integrators}
When we take  the following approximation for the stochastic integral
\begin{equation} \label{eq:stoch_integral}
 \SG {t_{n+1},t_n}  \int _{t_n} ^{t_{n+1}} \SG {s,t_n} ^{-1}  \vec{g}_i(\vec{u}(s))    dW_i (s) 
  \approx  \SG {t_{n+1},t_n}  \vec{g}_i(\vec{u}(t_n) )\Delta W_{i,n} 
\end{equation}
where $\Delta W_{i,n}=W_i(t_{n+1})-W_i(t_{n})$, we derive Euler type Exponential Integrators below. 
For the deterministic integral in \eqref{eq:exact} we examine three cases. 
%\begin{description}
%\item[\expint{EI0}] 
\begin{enumerate}
\item First taking
$\SG {t_{n+1},t_n} \int _{t_n} ^{t_{n+1}} \SG {s,t_n}^{-1} \tilde{\vec{f}} (\vec{u}(s))
ds \approx \SG {t_{n+1},t_n} \tilde{\vec{f}} (\vec{u}(t_n)) \Delta
t$, we obtain our first method \expint{EI0}
\begin{equation*}  \label{eq:EI0}
  \vec{u}_n=\SG{t_{n+1},t_n}\left( \vec{u}_n+ \tilde{\vec{f}} (\vec{u}_n) \Delta t +   \sum _{i=1} ^m \vec{g}_i(\vec{u}_n )\Delta W_{i,n} \right).
\end{equation*}
%\item[\expint{EI1}] 
\item If we take
$\SG {t_{n+1},s}  \int _{t_n} ^{t_{n+1}} \SG{s,t_n}^{-1} \tilde{\vec{f}} (\vec{u}(s))
ds \approx \ZZ{t_{n+1},t_n} \FF{\Dt A} \tilde{\vec{f}} (\vec{u}(t_n))
\Delta t $
where 
$$\ZZ{t,s}  =\exp \left( -\frac{1}{2} \sum_{i=1} ^m B_i ^2(t-s) +
  \sum_{i=1}^m B_i(W_i (t)  -W_i (s)) \right)
$$
then we obtain our second method \expint{EI1} 
\begin{equation*} \label{eq:EI1}
  \vec{u}_n=\SG{t_{n+1},t_n} \left(\vec{u}_n+ \sum _{i=1} ^m  \vec{g}_i(\vec{u}_n )\Delta W_{i,n} \right) + \ZZ{t_{n+1},t_n} \FF{\Delta t A}  \tilde{\vec{f}} (\vec{u}_n) \Delta t.
\end{equation*} 
%\item[\expint{EI2}] 
\item Finally with $\SG {t_{n+1},t_k}  \int _{t_n} ^{t_{n+1}} \SG {s,t_n}^{-1} \tilde{\vec{f}}
(\vec{u}(s))  ds \approx \FF{\Dt A} \tilde{\vec{f}} (\vec{u}(t_n))
\Delta t$ 
we get the method \expint{EI2}
\begin{equation*} \label{eq:EI2}
  \vec{u}_n=\SG{t_{n+1},t_n} \left(\vec{u}_n+ \sum _{i=1} ^m  \vec{g}_i(\vec{u}_n )\Delta W_{i,n} \right) + \FF{\Delta t A}  \tilde{\vec{f}} (\vec{u}_n)   \Delta t.
\end{equation*}
\end{enumerate}
%\end{description}
We compare the accuracy and efficiency of these approximations for
different numerical examples in \secref{sec:numexamples}.
In \secref{Sec:Mil} below we use a higher order approximation of  the stochastic
integral to derive Milstein versions of these scheme.
%Note that by employing higher order approximation for the stochastic
%integral rather than \eqref{eq:stoch_integral} it is possible to derive
%Milstein versions of these schemes.
For general noise the schemes \expint{EI0}, \expint{EI1}, \expint{EI2}
all have the same strong rate of convergence as \ETDint{SETD0} in
\eqref{eq:SETD0} and  \ETDint{SETD1} \eqref{eq:SETD1} which is $\Dt^{1/2}$.
However, we expect an improvement in the error when the terms in $B_i$
dominate $g_i$ in the noise. In the special case where $\vec{g}_i
\equiv 0$ we prove, and show numerically, an improvement in the strong
rate of convergence to order one.

It should be noted that all the proposed new type integrators
reduce to the usual exponential integrators \ETDint{SETD0} and
\ETDint{SETD1} when $B_i=0$, $i=1 \hdots m$. 
Indeed, it is observed  in numerical simulations that \ETDint{SETD}
schemes may perform better than the new  \expint{EI} schemes when $B_i$ are small compared to
$\vec{g}_i$. On the other hand the \expint{EI} schemes outperform
\ETDint{SETD} schemes when $B_i$ are dominant.
We can capture the good properties of both types of methods by
introducing a homotopy type parameter $p\in[0,1]$.
Let us rewrite \eqref{eq:EqABg} as 
\begin{equation} \label{eq:Eqhom}
  d\vec{u}= \left( A \vec{u}+\bf{F}(\vec{u}) \right) dt + \sum _{i=1} ^m  \left( p B_i \vec{u} +\vec{g}_i(\vec{u}) + (1-p) B_i \vec{u} \right)  dW_i (t) .
\end{equation}
For example, applying \expint{EI0} for this equation, one obtains \expint{HomEI0}
\begin{equation}  \label{eq:HomEI0}
  \vec{u}_n=\SG{t_{n+1},t_n} ^p\left( \vec{u}_n+ \tilde{\vec{f}} ^p (\vec{u}_n) \Delta t +   \sum _{i=1} ^m \vec{g}_i ^ p(\vec{u}_n )\Delta W_{i,n} \right)  
\end{equation}
where 
\begin{align} 
&\SG{t_{n+1},t_n} ^p=\exp \left(( A-\frac{1}{2} \sum_{i=1} ^m p^2 B_i ^2)\Delta t + \sum_{i=1}^m pB_i  \Delta W_i,n  \right),\\
&\vec{g}_i ^p (\vec{u})=\vec{g}_i(\vec{u}) + (1-p) B_i \vec{u},  \quad
  \text{and } \quad 
  \tilde{\vec{f}} ^p (\vec{u})= \vec{F} (\vec{u}) -\sum _{i=1} ^m p B_i \vec{g}_i ^p(\vec{u}).
\end{align}
It is clear that $p=0$ and $p=1$  give \ETDint{SETD0}  and
\expint{EI0}  respectively. In \secref{sec:numexamples} we
suggest a fixed formula for $p$ based on the weighting of $B_i$ to
$\vec{g}_i$. However, further consideration could be given to an
optimal choice of either a fixed $p$ or of a $p$ assigned
during the computation by considering weights of the terms 
in the diffusion coefficient, so that $p(u,B_i,g_i)$.
We note that unlike Milstein methods, \expint{HomEI0} and the other
\expint{EI} methods have the advantage that they do not require the
derivative of the diffusion term. 

\subsection{Milstein type Exponential Integrators}
\label{Sec:Mil}
An alternative treatment of \eqref{eq:exact} is to use the Ito-Taylor
expansion of the diffusion term
\begin{multline}
\label{eq:expansion}
\SG{s,t_n} ^{-1}  \vec{g}_i(\vec{u} (s))=\vec{g}_i(\vec{u} (t_n))+
\sum_{l=1}^m \int _{t_n} ^s  
\SG{r,t_n} ^{-1}   \vec{H}_{i,l} (\vec{u}(r)) dW_l(r)+\int _{t_n} ^s \SG{r,t_n} ^{-1}  \vec{Q}_i(\vec{u}(r)) dr
\end{multline}
where  
\begin{equation} \label{eq:H}
\vec{H}_{i,l} (\vec{u}(.))=D\vec{g}_i\left(\vec{u} \left(.\right)\right) \left( B_l \vec{u} \left(.\right)+ \vec{g}_l\left(\vec{u} \left(.\right)\right)\right) -B_l \vec{g}_i\left(\vec{u} \left(.\right)\right).
\end{equation}
and $\vec{Q}_i (.)$ is the  vector function in terms of
$A,\vec{F},D\vec{g}_i,D^2 \vec{g}_i$, $B_l$ for $i,l=1,...,m$ (which,
for ease of presentation, we do not detail here).

By freezing  the integrand of stochastic integral at $r=t_n$ and dropping the deterministic integral, one obtains the approximation
\begin{equation}
\SG{s,t_n} ^{-1}  \vec{g}_i(\vec{u} (s))= \vec{g}_i(\vec{u} (t_n))+\sum_{l=1}^m \int _{t_n} ^s   \vec{H}_{i,l} (\vec{u}(t_n)) dW_l(r)+h.o.t 
\end{equation}

Using this approximation, we  obtain the Milstein scheme \expint{\Mil}
\begin{multline}  \label{eq:MilEI0}
  \vec{u}_{n+1}=\SG{t_{n+1},t_n}\left( \vec{u}_n+ \tilde{\vec{f}}
    (\vec{u}_n) \Delta t +   \sum _{i=1} ^m \vec{g}_i(\vec{u}_n
    )\Delta W_{i,n} \right.\\ 
\left. +\sum_{i=1}^m \sum_{l=1}^m    \vec{H}_{i,l} (\vec{u}_n) \int_{t_n}^{t_{n+1}} \int_{t_n}^{s} dW_l(r)dW_i(s) \right).
\end{multline}
We can also introduce a Milstein homotopy type scheme \expint{HomMI0}
by applying \expint{\Mil} to \eqref{eq:Eqhom}.

\section{Convergence result and numerical examples}
\label{sec:Mainresults}
We state in this section the strong convergence result for
both \ETDint{EI0} and \ETDint{\Mil}. Proofs are given in
\secref{sec:proofs} and we note that the proofs for the other schemes,
including those such as \eqref{eq:HomEI0}, are similar.  
For these proofs we assume a global Lipschitz condition on the drift and
diffusion. Tamed version of the methods for more general drift and
diffusions can be derived \cite{ErdoganLordTaming}. We let $\|\cdot\|_2$ denote the standard Euclidean norm and $\Lnorm{\cdot}^2=\eval{\norm{\cdot}^2}$.

\begin{assumption} 
  \label{ass:1}
  There exists a constant $L>0$ such that the linear growth condition
  holds: for $\vec{u} \in \mathbb{R}^d $ and $i=1,\hdots, m$
  $$
  \norm{\vec{F}(\vec{u})}^2 \leq  L(1+\norm{\vec{u}}^2), \qquad 
  \norm{\vec{g}_i (\vec{u})}^2 \leq  L(1+\norm{\vec{u}}^2), %\quad i=1,\hdots, m
  $$
  and the global Lipschitz condition holds: for $\vec{u},\vec{v} \in
  \mathbb{R}^d$, $i=1,\hdots, m$
  $$
  \norm{\vec{F}(\vec{u})-\vec{F}(\vec{v})} \leq L\norm{\vec{u}-\vec{v}
  }, \qquad 
  \norm{\vec{g}_i(\vec{u})-\vec{g}_i(\vec{v})} \leq L
  \norm{\vec{u}-\vec{v}}   
  %\quad i=1,\hdots, m  
  $$
\end{assumption}
First we state the strong convergence result for the Euler type  scheme
\expint{EIO}.
\begin{theorem*}
\label{thrm:1}
Let  Assumptions \ref{ass:1}  hold and let $\vec{u}_n$
be  approximation to the solution of \eqref{eq:EqABg} using
\expint{EI0}. For $T>0$, there exists $K>0$ such that 
\begin{equation}
  \sup_{0\leq t_n\leq T} \Lnorm{\vec{u}(t_n) - \vec{u}_n}  \leq K \Delta t ^{1/2}.
\end{equation} 
\end{theorem*}
For the Milstein scheme \expint{\Mil}, we impose the following two
extra assumptions.
\begin{assumption} 
  \label{ass:2}
  The functions  $ F,\vec{g}_i: \mathbb{R}^d \rightarrow
  \mathbb{R}^d $  are twice  continuously  differentiable.  
\end{assumption} 
\begin{assumption} 
  \label{ass:3}
  For the same constant $L$ as in Assumption \ref{ass:1} 
  for $\vec{u},\vec{v} \in \mathbb{R}^d$, $i,l=1,\hdots, m$
  $$
  \norm{D\vec{g}_i(\vec{u})\vec{g}_l(\vec{u})-D\vec{g}_i(\vec{v})\vec{g}_l(\vec{v}) } \leq L
  \norm{\vec{u}-\vec{v}}
  %\quad i,l=1,\hdots, m  
  $$
and  
   $$
  \norm{D\vec{g}_i(\vec{u}) B_l \vec{u} - D\vec{g}_i(\vec{v})B_l \vec{v} } \leq L
  \norm{\vec{u}-\vec{v}}.
  %\quad i,l=1,\hdots, m  
  $$
\end{assumption}

\begin{theorem*}
\label{thrm:2}
Let  Assumptions \ref{ass:1}, \ref{ass:2} and \ref{ass:3} hold and let
$\vec{u}_n$ 
be  approximation to the solution of \eqref{eq:EqABg} using
\expint{\Mil}. For $T>0$, there exists $K>0$ such that 
\begin{equation}
  \sup_{0\leq t_n\leq T} \Lnorm{\vec{u}(t_n) - \vec{u}_n}  \leq K \Delta t. %^{1}.
\end{equation} 
\end{theorem*}

Note that from the definition of $\tilde{\vec{f}}$ in
\eqref{eq:def_of_functions} and $\vec{H}_{i,l}$ in \eqref{eq:H}, these functions also  satisfy  global Lipschitz
and/or  continuously differentiability  conditions when  the corresponding assumptions on $\vec{F}$, $\vec{g}_i$ and $D\vec{g}_i$ hold. 
We give the proofs of both these Theorems in \secref{sec:proofs}. 

Now consider the special case when $\vec{g}_i\equiv0$  in 
\eqref{eq:EqABg}. Namely, we have the SDE  
\begin{equation} \label{eq:Eq0}
    d\vec{u}= \left( A \vec{u}+\bf{F}(\vec{u}) \right) dt + \sum
    _{i=1} ^m   B_i \vec{u}   dW_i (t), \qquad \vec{u} (0)=\vec{u}_0 \in  \mathbb{R} ^d
\end{equation}   
for which both the numerical schemes \expint{EI0} and \expint{\Mil} reduce to %\expint{EIF0}
\begin{equation}  \label{eq:EIF0}
  \vec{u}_n=\SG{t_{n+1},t_n}\left( \vec{u}_n+ \vec{F} (\vec{u}_n) \Delta t \right).
\end{equation}
Remark that we can consider \eqref{eq:EIF0} as a Lie Trotter splitting of
\eqref{eq:Eq0}. 
It is straightforward to conclude the following improvement in the
convergence rate for  \expint{EI0}.
\begin{corollary}
\label{cor:1}
Let  Assumption \ref{ass:1}  and continuously differentiability condition hold for
$\vec{F}$ and let  $\vec{u}_n$ denote the 
approximation to the solution of \eqref{eq:Eq0} by \eqref{eq:EIF0}. 
For $T>0$, there exists $K>0$ such that 
\begin{equation}
\sup_{0\leq t_n\leq T} \Lnorm{\vec{u}(t_n) - \vec{u}_n}  \leq K \Delta t .
\end{equation}
\end{corollary}
This is a simple consequence of solving the linear SDE exactly,
see \secref{sec:proofs}. 
 
\subsection{Numerical examples}
\label{sec:numexamples}
In this section we perform some numerical experiments to illustrate
and confirm the orders of the  proposed methods. For comparison
\ETDint{SETD0}, \ETDint{SETD1}, Exponential Milstein  \ETDint{ExpMIL}
\cite{expmil}, the classical Milstein \cite{Kloedenbook} are used as well as 
%as they have been shown to be more effective
%than 
the semi-implicit Euler--Maruyama scheme (EM). 
%employed because  of  their effectiveness compared to semi implicit Euler Murayama scheme.  

\subsection*{Example 1: Ginzburg-Landau Equation} 
Consider  the  one dimensional equation
\begin{equation}
\label{eq:SGL}
du(t)= \left( - u +\frac{\sigma}{2}u - u^3 \right)  dt+\sqrt{\sigma} u dW(t), \qquad u(0)=u_0
\end{equation}
that has exact solution \cite{Kloedenbook}
\begin{equation}
u(t)=\frac{u_0e^{-t+\sqrt{\sigma} W(t)} }{\sqrt{1+2u_0^2 \int _0 ^t  e^{-2s+2\sqrt{\sigma} W(s)}}}.
\end{equation}
%This exact  solution is employed  to compare  performance  of the
%proposed methods with \expint{SETD1} and Milstein type Exponential
%Integrator \cite{expmil}. However, i
It  should be noted that the drift term satisfies only  a one sided
global Lipschitz condition and our proposed schemes might need to be
tamed to guarantee strong convergence as in
\cite{hutzenthaler2015numerical}. Analysis of taming for these
schemes is considered in \cite{ErdoganLordTaming}.
Nevertheless, ordinary Monte Carlo simulations reveal the performance of
the new schemes and act as a benchmark for \ETDint{SETD1} (see also
\cite{expmil}). 
In this SDE \expint{\Mil} and \expint{HomMI0} both reduce to
\expint{EI0} and \expint{HomEI0}. We compare here the schemes 
\expint{EI0}, \expint{EI1}, \expint{EI2} and \expint{HomEI0}.
Note that \eqref{eq:SGL} is linear in the diffusion and
hence Corollary \ref{cor:1} holds and we expect first order convergence.
This is observed in \figref{fig:SGL} (a) where we see first order
convergence of the methods \expint{EI0}, \expint{EI1} and
\expint{EI2}. In \figref{fig:SGL} (b) we compare the efficiency of
the schemes and observe that \expint{EI0} is the most efficient.
For the other examples that we consider we now only show results for
\expint{EI0} and \expint{HomEI0}.
\begin{figure}[ht]
\begin{center}
(a) \hspace{.49\textwidth} (b) \\
\includegraphics[width=0.49\textwidth]{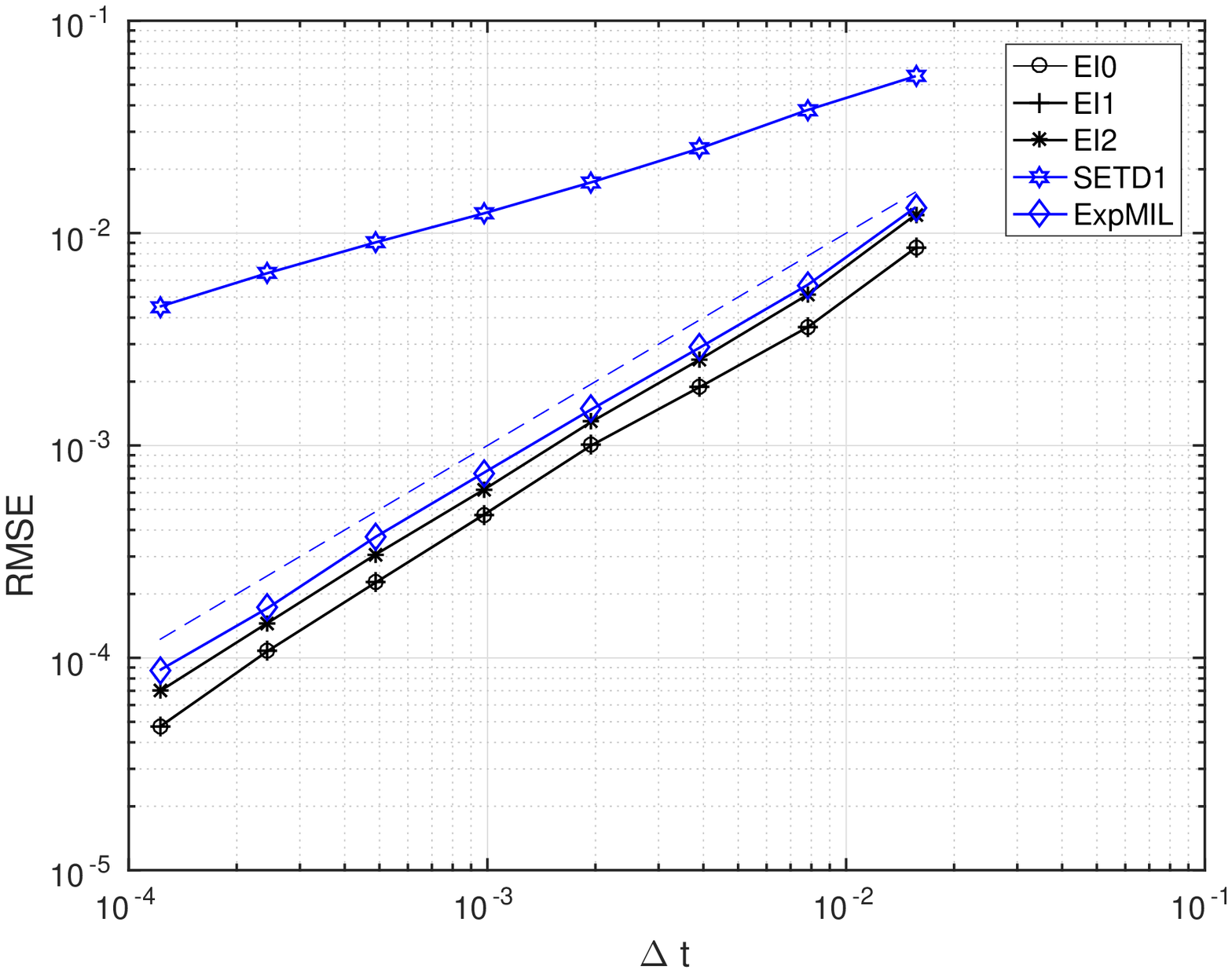}
\includegraphics[width=0.49\textwidth]{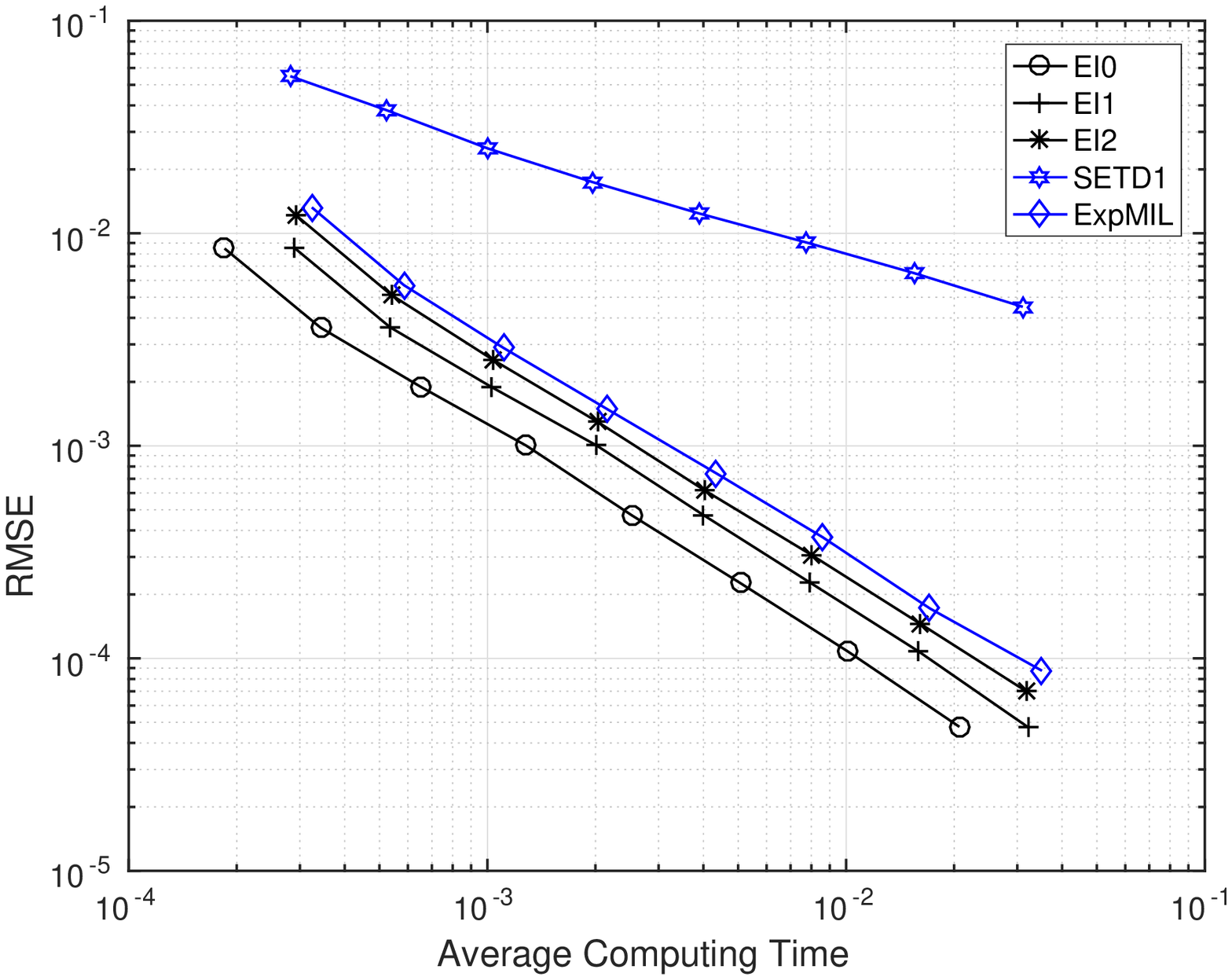}
\caption{Stochastic Ginzburg-Landau Equation \eqref{eq:SGL} with
  $\sigma=2$, $T=1$, $M=1000$ samples (a) root mean square error
  against $\Dt$. Also plotted is a reference line with slope 1.
  (b)  root mean square error against cputime. Of the
  new schemes \expint{EI0} is the most efficient.
  We observe the improved convergence rate of Corollary \ref{cor:1} in these new
  schemes over that for \ETDint{SETD1}.
}
\label{fig:SGL}
\end{center}
\end{figure}

\subsection*{Example  2: nonlinear and non-commutative noise}
Consider the following SDE in $\mathbb{R}^4$, with initial data $\vec{u}(0)= \left(1,1,1,1 \right) ^T$
\begin{equation} 
  \label{eq:SPDEexample}
    d\vec{u}= \left( r A \vec{u}+\bf{F}(\vec{u}) \right) dt +
    G(\vec{u}) d \vec{W}(t), \qquad  F_j=\frac{u_j}{1+\vert u_j
      \vert}, 
%\quad \vec{u}(0)= \left(1,1,1,1 \right) ^T,
\end{equation}
where $r$ is a constant (we take $r=4$) and $A$ arises from the
standard finite difference approximation of the Laplacian
\begin{equation}
  A=\begin{pmatrix}
    -2 & 1 & 0 & 0 \\
    1&-2&1&0\\
    0&1&-2&1  \\
    0&0&1&-2
  \end{pmatrix}. \
\end{equation}
As we do not have an exact solution in this
example we compute a reference solution using the exponential
Milstein method with a small step size $\Dt_{\text{ref}}$
%$=1/2^{20}$ (or $2^{14}$ for the non-commutative noise examples) 
and examine a Monte Carlo estimate of the
error $\Lnorm{\vec{u}(t_n) - \vec{u}_n}$ with $M$ %=1000$ 
realisations.% ($M=100$ for non-commutative noise examples).

%This matrix is studied in finite difference   discretization of
%Laplacian. We  will consider commutative  and  non commutative noise
%terms separately. 
%
\subsubsection*{Diagonal Noise}
First we look at diagonal noise and examine the effective of the noise
being dominated by either linear or nonlinear terms. 
For the nonlinear part we let  $g(u)=1/(1+u^2)$ and let $\vec{g}_i (\vec{u})$ have only one 
non-zero element $\alpha g(u_i)$ in the ith  entry for $\alpha
\in\mathbb{R}$.
For the linear part we take $B_i=\beta diag(\vec{e}_i)$ where
$\vec{e}_i$ is the ith unit vector of $\mathbb{R} ^4$ and $\beta\in\mathbb{R}$.
This gives $G(u)$ in \eqref{eq:SPDEexample} as
\begin{equation*}
  \label{eq:diagexample}
  G(\vec{u})=\begin{pmatrix}
    \beta u_1 +  \alpha g(u_1) & 0 & 0 \\
    0& \beta u_2 + \alpha  g(u_2) &0&0\\
    0&0& \beta u_3 + \alpha  g(u_3)&0 \\
    0&0&0&\beta u_4 + \alpha  g(u_4)
  \end{pmatrix} \ .
\end{equation*}
When $\alpha << \beta$ the linear terms $B_i$ dominate, whereas if
$\alpha >> \beta$, the nonlinearity $\vec g_i$ dominates.
By examining different $\alpha$ and $\beta$ we can see the effect of
the strength of the nonlinearity.
We take $\Dt_{\text{ref}}=2^{-20}$ and $M=1000$.
For \expint{HomEI0} and \expint{HomMI0} we define the homotopy parameter by
\begin{equation}
p=\dfrac{\mid \beta  \mid}{\mid \alpha \mid +\mid \beta \mid}.
\label{eq:hompar}
\end{equation}
A  matlab script to implement \expint{HomEI0} is presented in 
Algorithm \ref{alg:1}. 
We show results for the both the Euler and Milstein type schemes in
each case.
\begin{algorithm}
  \caption{Matlab script to solve \eqref{eq:SPDEexample} with noise
    given by \eqref{eq:diagexample} 
    using \expint{HomEI0} }
  \label{alg:1}
  \begin{lstlisting} 
  N=pow2(10);T=1.0;Dt=T/N;%number of steps,final time,step size
  d=4;m=4; % dimension of problem and dimension of noise  
  r=4;beta=1;alpha=0.1; % parameters  of the problem
  X=ones(d,1); %initial Condition
  p=abs(beta)/(abs(beta)+abs(alpha));%set homotopy parameter
  % Set Matrices 
  A=-r*sparse(toeplitz([2 -1 zeros(1, d-2)])); M1=expm(Dt*A); 
  % Set functions
  f=@(u) u./(1+abs(u)); g=@(u) 1./(1+u.^2);
  ftilde=@(u) f(u)-p*beta*(alpha*g(u)+(1-p)*beta*u);
  Gtilde=@(u) sparse(diag(alpha*g(u)+ (1-p)*beta*u));
  for n=1:N   % loop over time steps
    dW = sqrt(Dt)*randn(m,1);  % get increment for noise
    M2 = exp(-Dt*0.5*p^2*beta^2+p*beta*dW);
    X=M1*M2.*(X+Dt*ftilde(X)+Gtilde(X)*dW);  % update step
  end
\end{lstlisting}
\end{algorithm}
First consider the case where $\alpha=0.1$ and $\beta=1$ so that the
linear term dominate. \figref{fig:2} (a)  illustrates  orders and (b)
the efficiency  of the methods \expint{EI0}, \ETDint{SETD0}, \expint{HomEI0}.
In \figref{fig:2} (a) we see convergence with the predicted rate and
in \figref{fig:2} (b) it is clear that \expint{EI0} and
\expint{HomEI0} are more efficient than either \ETDint{SETD0} or the
semi-implicit Euler--Maruyama method (EM). (Recall that if $\beta=0$ then
we obtain first order convergence for  \expint{EI0} and
\expint{HomEI0} which is not the case for \ETDint{SETD0} or EM).
\figref{fig:2MIL} (a) shows first order convergence for the Milstein
schemes and from (b) we see that \expint{HomMI0} and \expint{MI0} are
the most efficient.
\begin{figure}[ht]
\begin{center}
(a) \hspace{.49\textwidth} (b) \\
\includegraphics[width=0.49\textwidth]{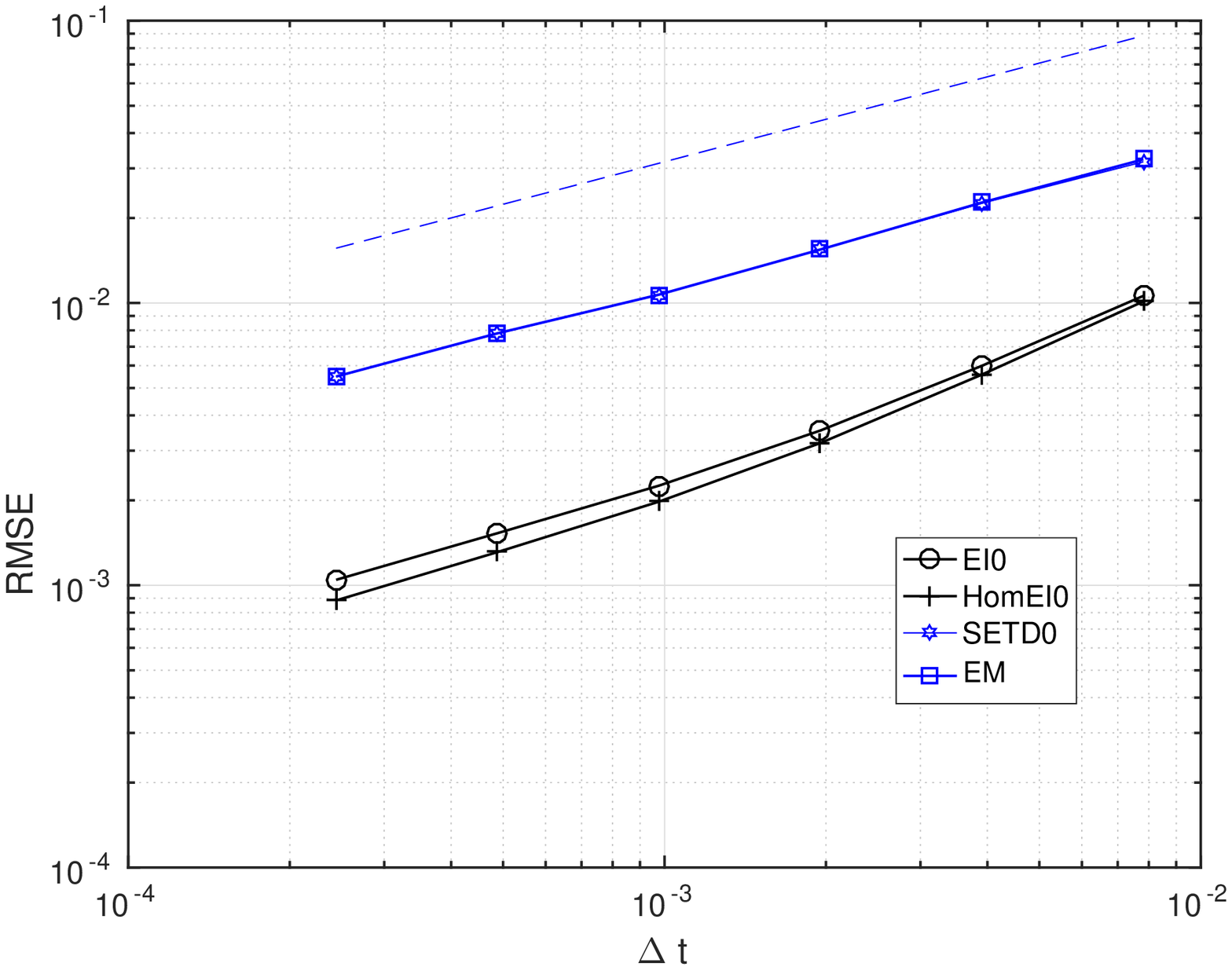}
\includegraphics[width=0.49\textwidth]{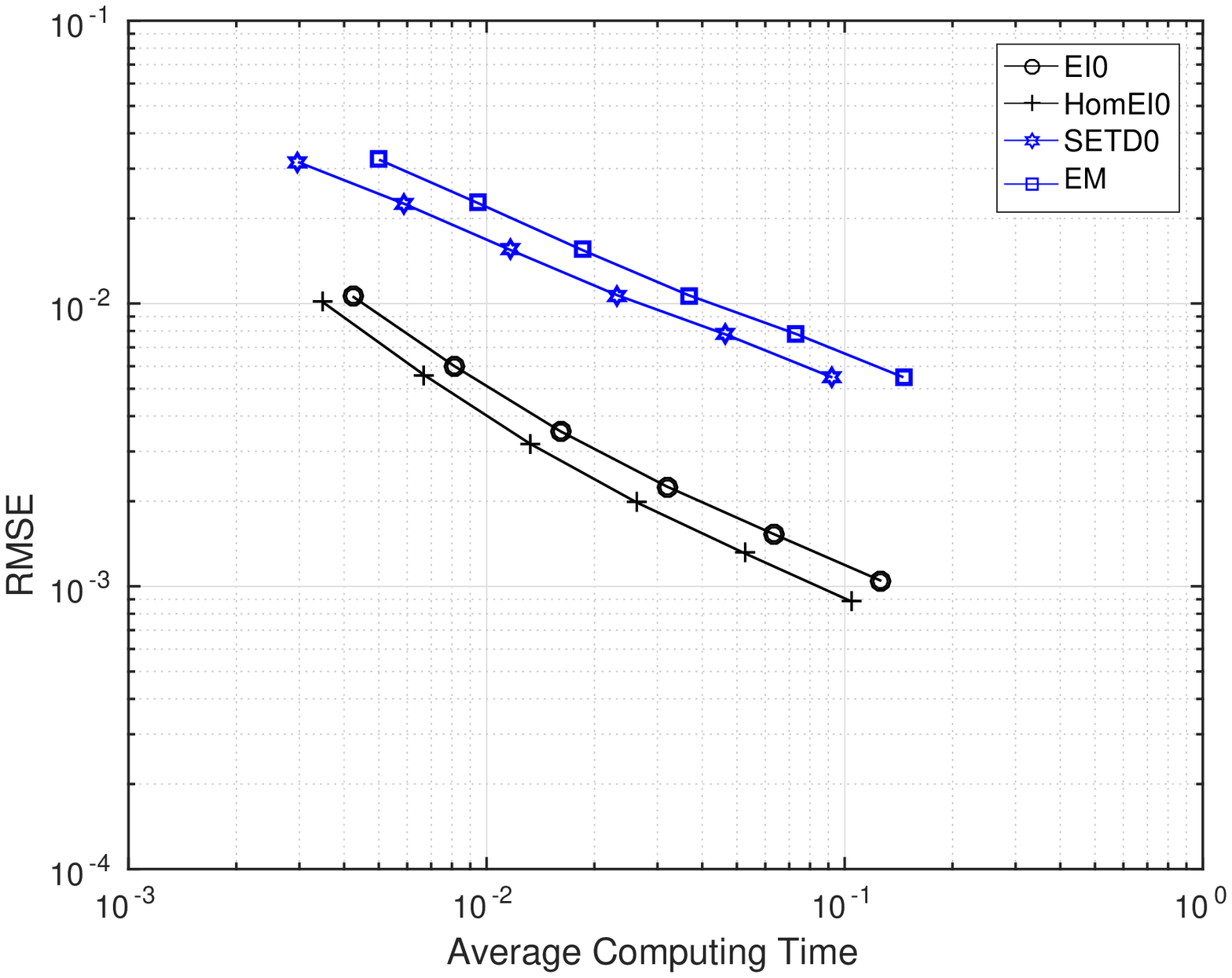}
\caption{Euler methods. Equation \eqref{eq:SPDEexample} 
  $\beta=1$, $\alpha=0.1$, $r=4$, $T=1$ and $M=1000$ samples (a) root
  mean square error 
  against $\Dt$. Also plotted is a reference line with slope 1/2.
  (b)  root mean square error against cputime. Here the linear noise
  term dominates and we see \expint{HomEI0} is the most efficient,
  followed by \expint{EI0}. See \figref{fig:2MIL} for Milstein schemes.}
\label{fig:2}
\end{center}
\end{figure}
\begin{figure}[ht]
\begin{center}
(a) \hspace{.49\textwidth} (b) \\
\includegraphics[width=0.49\textwidth]{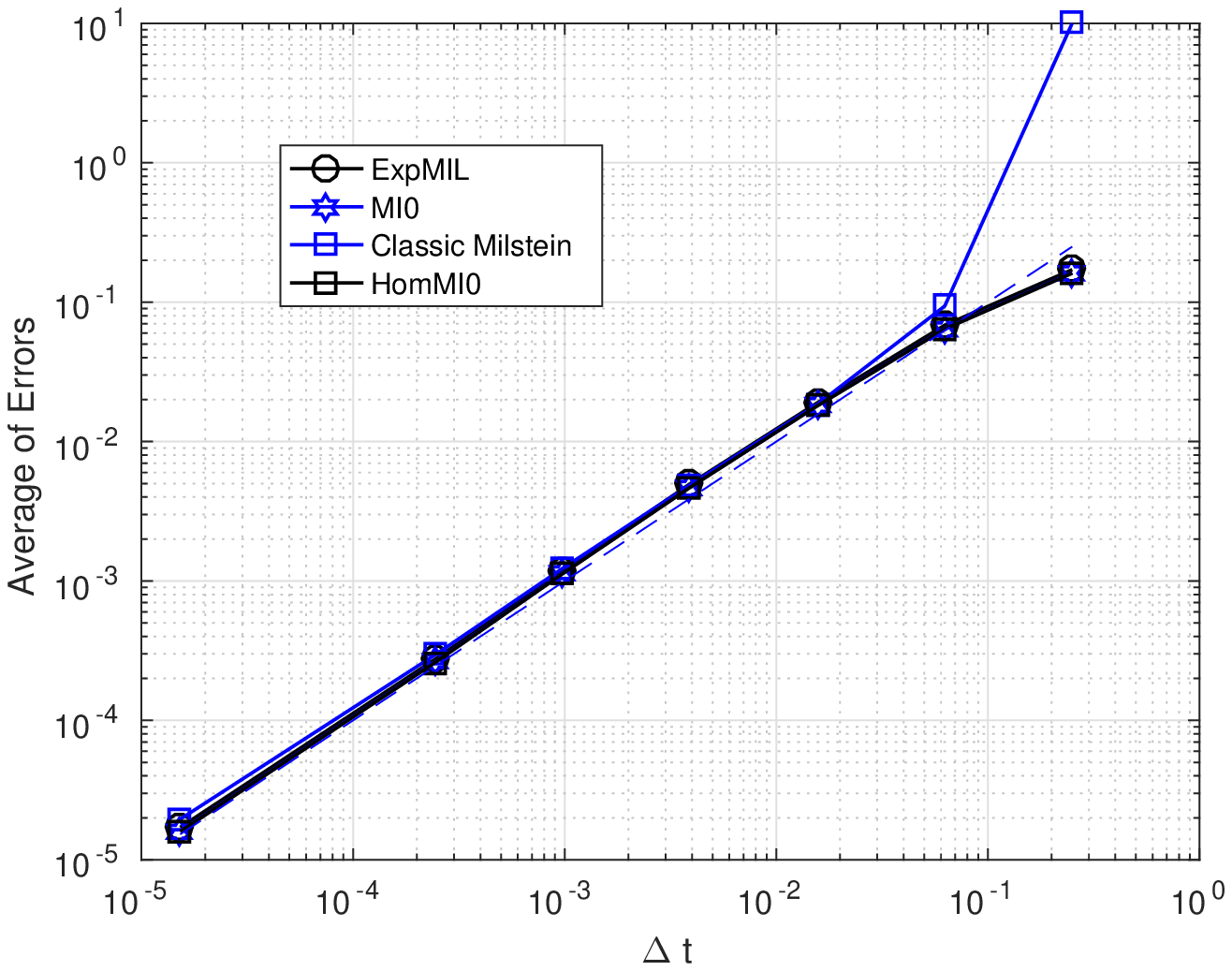}
\includegraphics[width=0.49\textwidth]{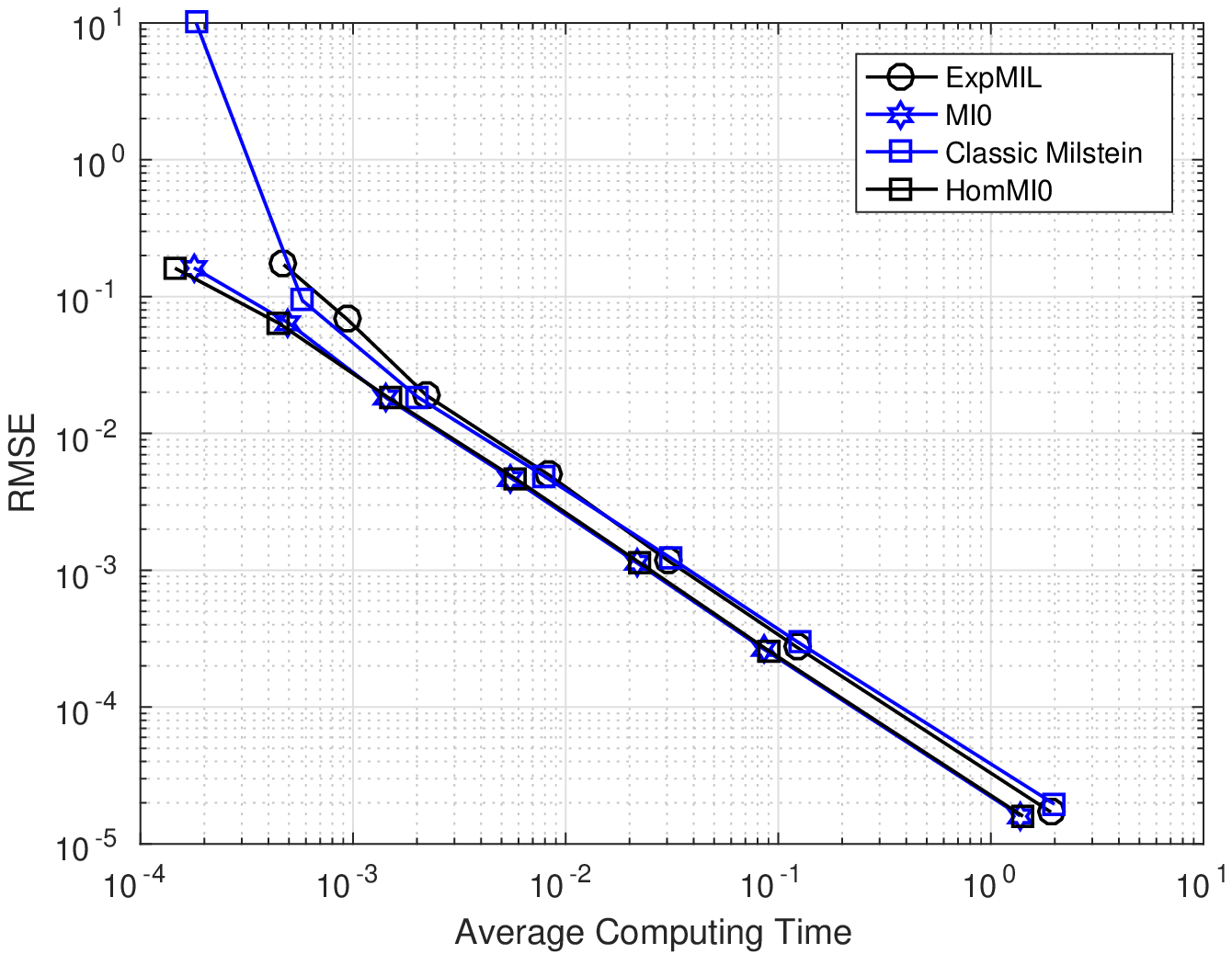}\\
\caption{Milstein methods. Equation \eqref{eq:SPDEexample} with 
  $\beta=1$, $\alpha=0.1$, $r=4$, $T=1$ and $M=100$ samples (a) root
  mean square error against $\Dt$. Also plotted is a reference line
  with slope 1 (compare to \figref{fig:2}). In 
  (b) root mean square error against cputime. Here the linear noise
  term dominates and we see \expint{HomMI0} and \expint{MI0} are the
  most efficient.}
\label{fig:2MIL}
\end{center}
\end{figure}
However, when $\beta=\alpha=1$ where we have equal weighting between
the linear and nonlinear term we see in \figref{fig:3} (a) the same
rate of convergence but now \ETDint{SETD0} and EM are more accurate
than \expint{EI0}. For efficiency we see in \figref{fig:3} (b) that
\expint{HomEI0} is still the most efficient, followed by
\ETDint{SETD0}. This illustrates the effectiveness of adding the
homotopy parameter. For the Milstein schemes we see the predicted
rate of convergence in \figref{fig:3MIL}  (a) and in (b) that 
\expint{HomEI0} and \expint{MI0} are marginally more efficient than
either the classical Milstein or Exponential Milstein schemes.
\begin{figure}[ht]
  \begin{center}
    (a) \hspace{.49\textwidth} (b) \\
    \includegraphics[width=0.49\textwidth]{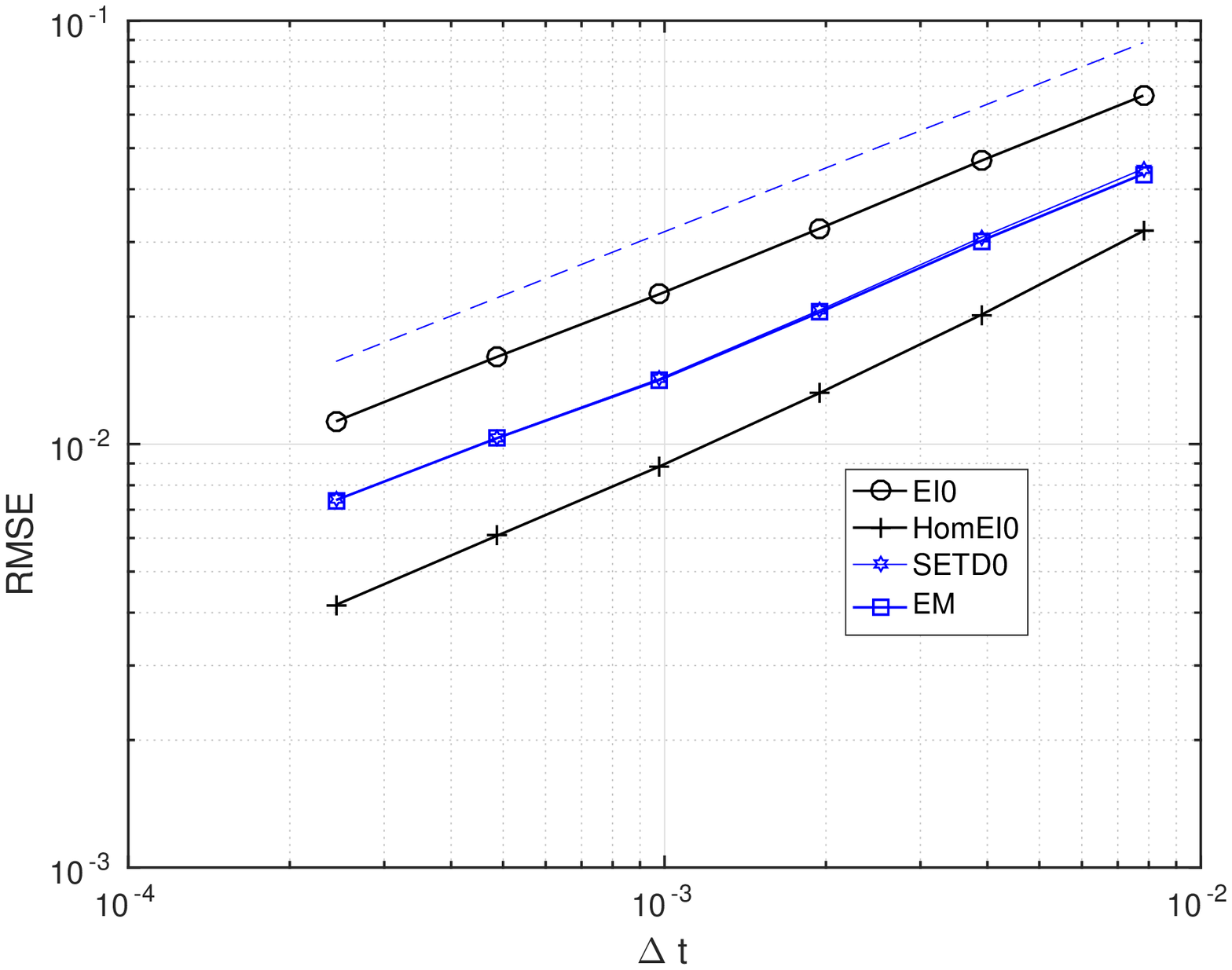}
    \includegraphics[width=0.49\textwidth]{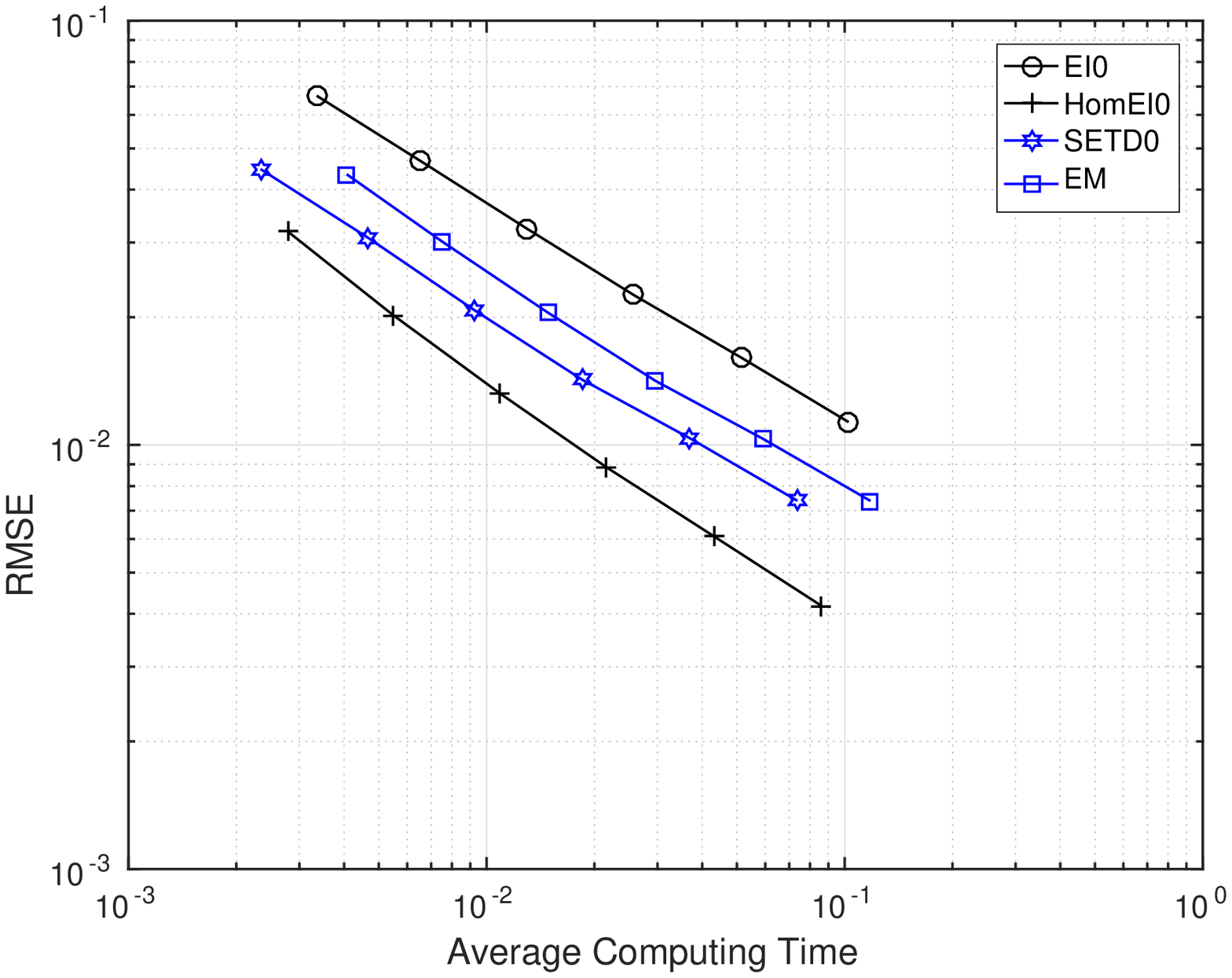}\\
    \caption{Euler methods. Equation \eqref{eq:SPDEexample} with 
      $\beta=1$, $\alpha=1$, $r=4$, $T=1$ and $M=1000$ samples (a) root
      mean square error against $\Dt$. Also plotted is a reference line
      with slope 1/2. (b)  root mean square error against cputime. We have
      equal weighting of linear and nonlinear noise terms
      and we see \expint{HomEI0} is clearly the most efficient and
      accurate. See \figref{fig:3MIL} for Milstein type schemes.}
    \label{fig:3}
  \end{center}
\end{figure}
\begin{figure}[ht]
  \begin{center}
    (a) \hspace{.49\textwidth} (b) \\
    \includegraphics[width=0.49\textwidth]{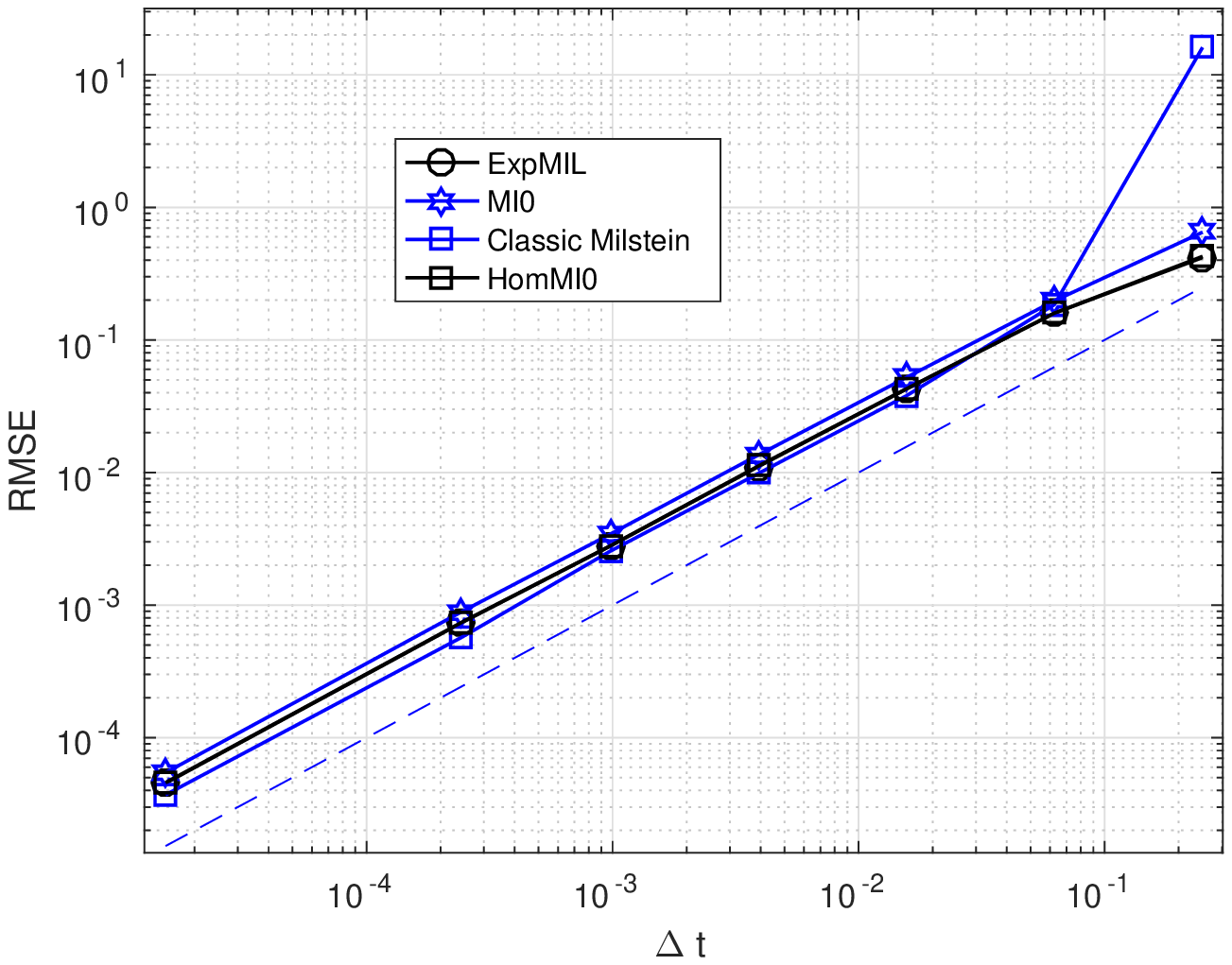}
    \includegraphics[width=0.49\textwidth]{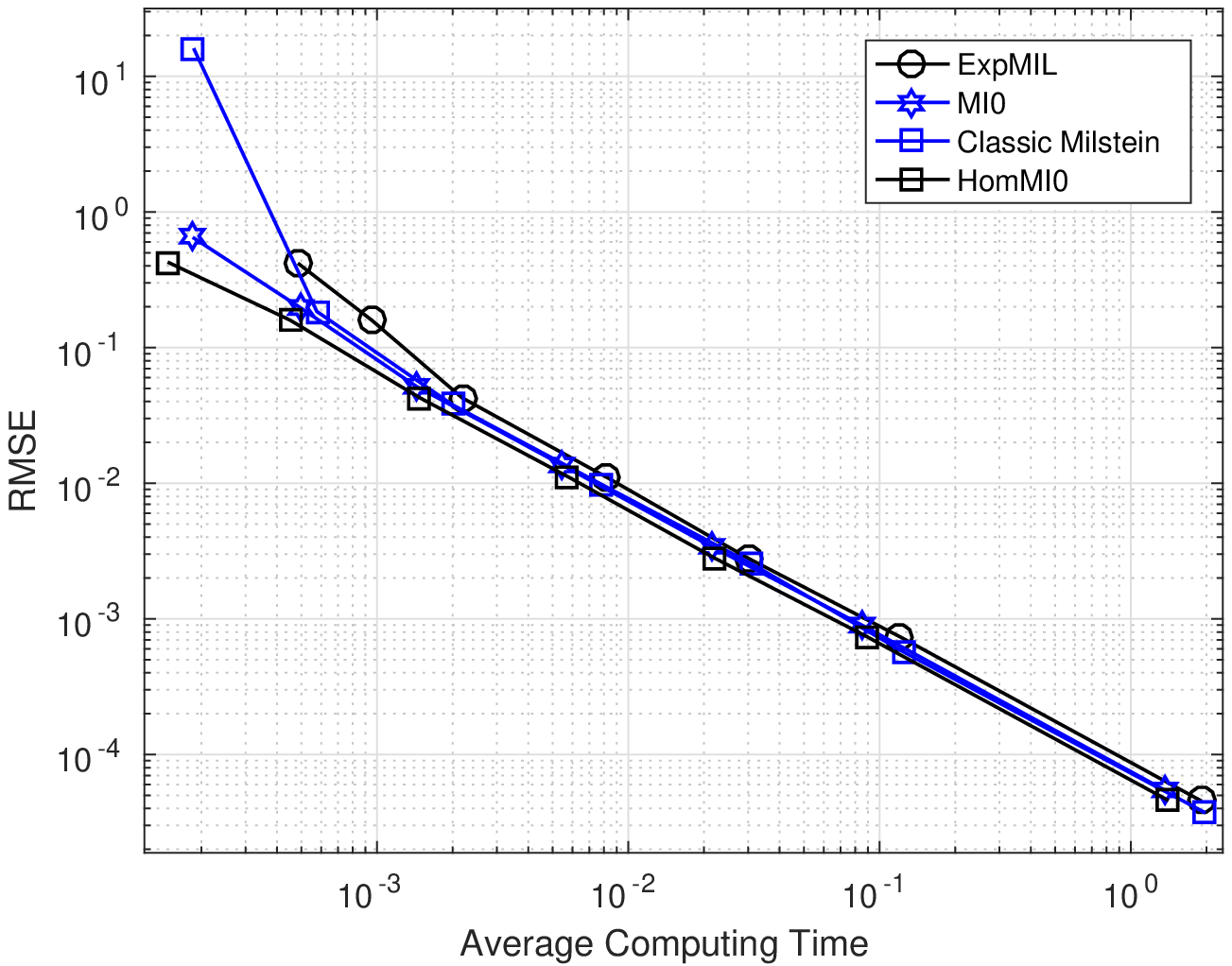}\\
    \caption{Milstein methods. Equation \eqref{eq:SPDEexample} with 
      $\beta=1$, $\alpha=1$, $r=4$, $T=1$ and $M=100$ samples (a) root
      mean square error against $\Dt$. Also plotted is a reference line
      with slope 1 (compare to \figref{fig:3}). In (b) root mean
      square error against cputime. With equal weighting of the noise
      we see \expint{HomMI0} and \expint{MI0} are marginally more efficient.}
    \label{fig:3MIL}
  \end{center}
\end{figure}
Next we consider in \figref{fig:3cd} the case where $\beta=1$ and
$\alpha=0.1$ so that it is the nonlinearity that dominates. We now see
that the errors from \expint{HomEI0} are similar to those or the standard
integrators \ETDint{SETD0} and EM and that \ETDint{SETD0} is now more
efficient. We note, however, that \expint{HomEI0} remains more
efficient than EM.
For the Milstein schemes we see the predicted
rate of convergence in \figref{fig:3cdMIL} (a) and in (b) that 
\expint{HomEI0} and \expint{MI0} are more efficient than
either the classical Milstein or Exponential Milstein schemes.
\begin{figure}[ht]
  \begin{center}
    (a) \hspace{.49\textwidth} (b) \\
    \includegraphics[width=0.49\textwidth]{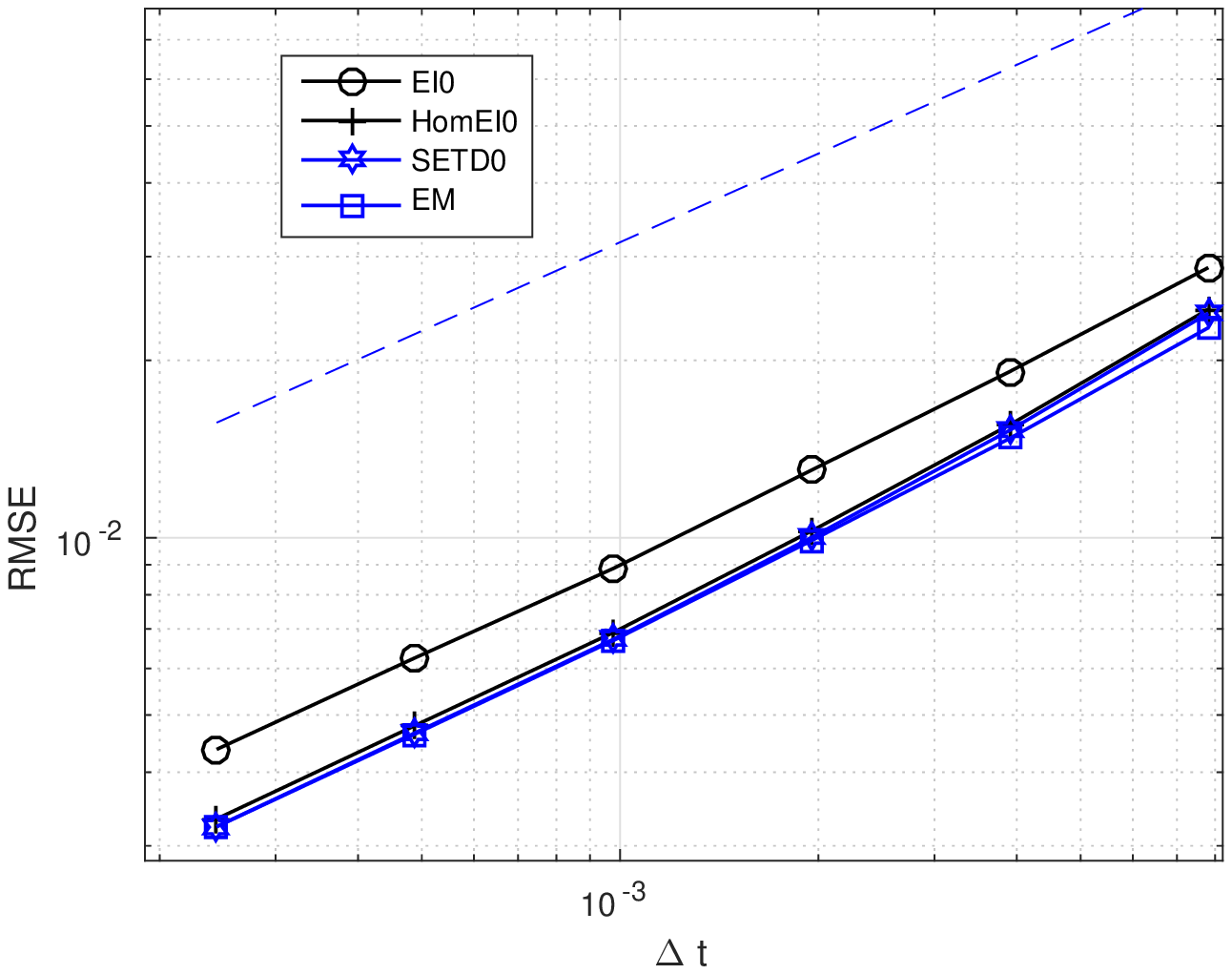}
    \includegraphics[width=0.49\textwidth]{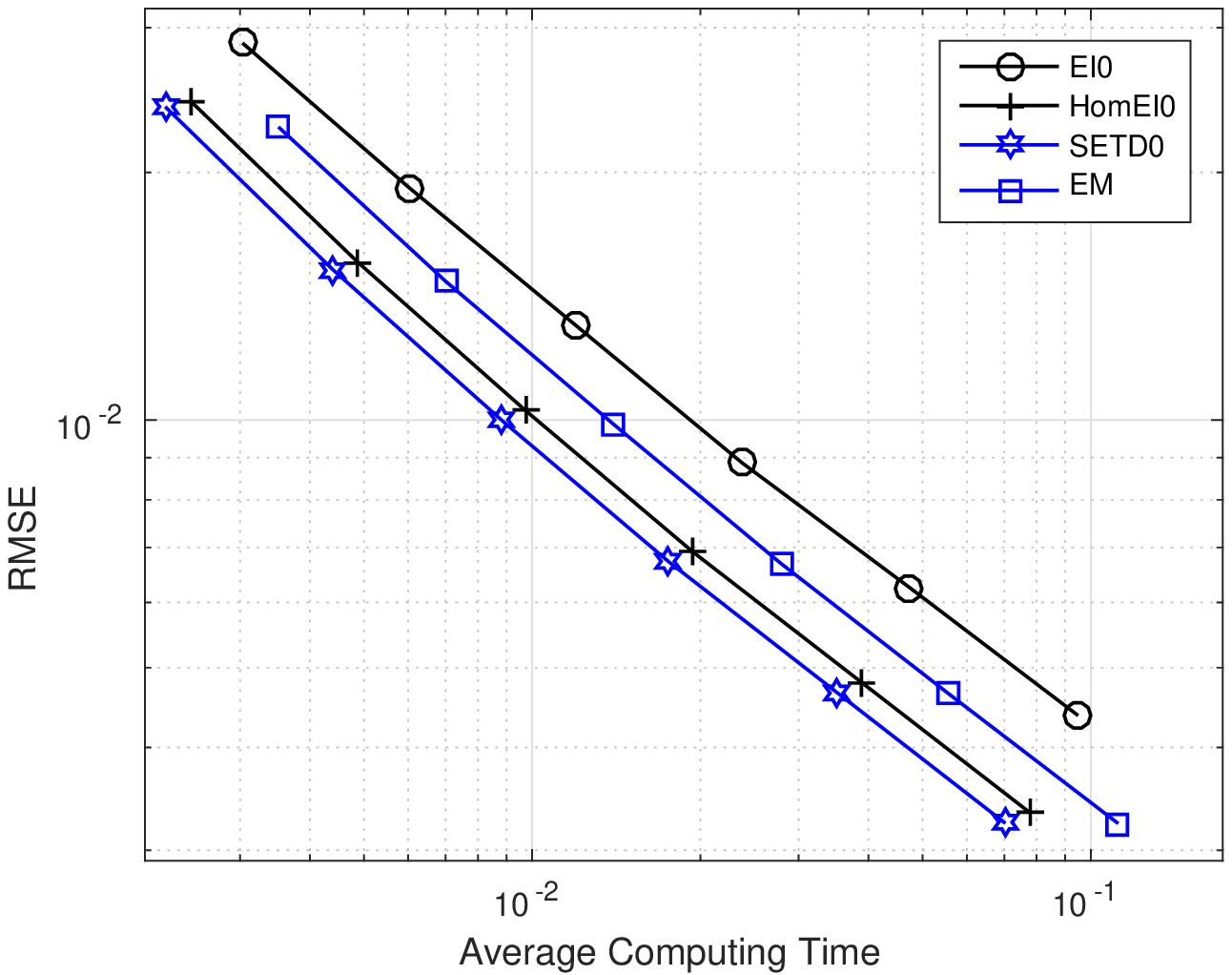}
    \caption{Euler methods. Equation \eqref{eq:SPDEexample} with 
      $\beta=0.1$, $\alpha=1$, $r=4$, $T=1$ and $M=1000$ samples (a) root
      mean square error against $\Dt$. Also plotted is a reference line
      with slope 1/2. (b)  root mean square error against cputime.
      The noise is dominated by the nonlinear term. We now see that
      \expint{SETD0} is the most efficient, followed by
      \expint{HomEI0}. See \figref{fig:3cdMIL} for Milstein type schemes.}
    \label{fig:3cd}
  \end{center}
\end{figure}
\begin{figure}[ht]
  \begin{center}
    (a) \hspace{.49\textwidth} (b) \\
    \includegraphics[width=0.49\textwidth]{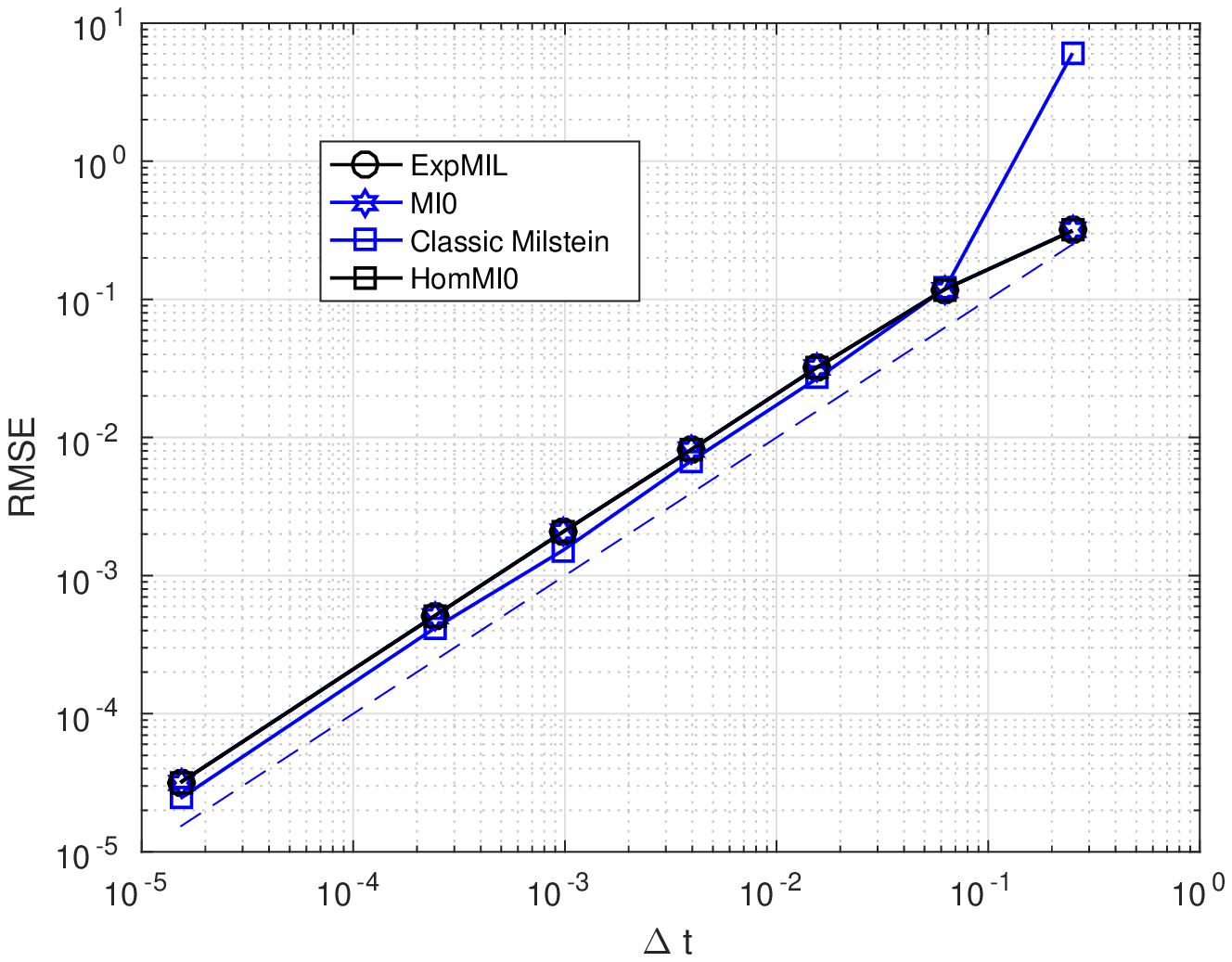}
    \includegraphics[width=0.49\textwidth]{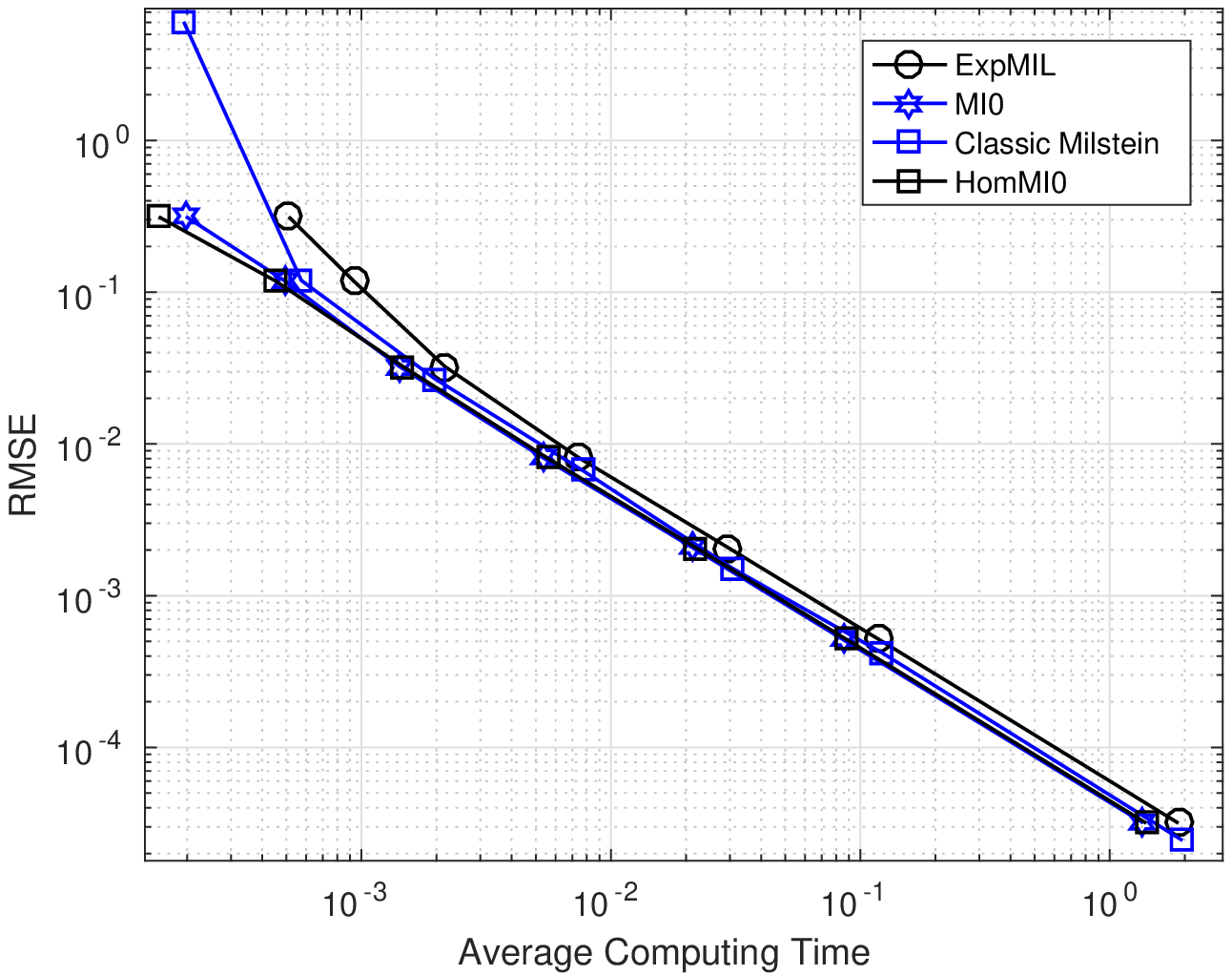}\\
    \caption{Milstein methods. Equation \eqref{eq:SPDEexample} with 
      $\beta=0.1$, $\alpha=1$, $r=4$, $T=1$ and $M=100$ samples (a) root
      mean square error against $\Dt$. Also plotted is a reference line
      with slope 1 (compare to \figref{fig:3cd}). In 
      (b) root mean square error against cputime. The noise is
      dominated by the nonlinear term. We see that \expint{HomMI0} and \expint{MI0} are the
  most efficient.}
    \label{fig:3cdMIL}
  \end{center}
\end{figure}

\subsection*{Non Commutative  Noise}
Now consider \eqref{eq:SPDEexample} with non-commutative noise by
taking the following diffusion coefficient matrix
\begin{equation}
G(\vec{u})=\begin{pmatrix}
\beta u_1 & 0 & 0 \\
0&\beta u_2 - \alpha u_1 &0&0\\
0&0&\beta u_3 - \alpha u_2&0 \\
0&0&0&\beta u_4 - \alpha u_{3}
\end{pmatrix} \
\end{equation} 
In order to apply  \expint{EI} schemes, consider  the splitting
\begin{equation}
G(\vec{u})= \beta \begin{pmatrix}
 u_1 & 0 & 0 & 0  \\
0& u_2  &0&0\\
0&0&u_3 &0 \\
0&0&0& u_4 
\end{pmatrix} \
- \alpha  \begin{pmatrix}
0 & 0 & 0 & 0  \\
0&  u_1 &0&0\\
0&0& u_2&0  \\
0&0&0& u_{3}
\end{pmatrix} \
\end{equation}  
that gives  the matrices $B_i=\beta diag(\vec{e}_i)$   and the vectors
$\vec{g}_i (\vec{u})$ having only non zero  element $-\alpha u_i $ in
$(i-1)$ th  entry. 
In this case Levy areas are now needed to apply the exponential
Milstein scheme and due to this extra computational cost in obtaining a
reference solution we reduce the number of samples to $M=100$
and take $\Dt_{\text{ref}}=2^{-14}$

\figref{fig:45} compares the cases where $\beta=1$, $\alpha=0.1$ in
(a) and (b) and $\beta=1$, $\alpha=1$ in (c) and (d).
When the linear term dominates \figref{fig:45} (a) and (b) we see that
the schemes \expint{HomEI0} and \expint{EI0} have smaller error and are
the most efficient.  In  \figref{fig:45} (c) and (d), where there is
an equal weighting between the diagonal and nondiagonal term in the noise,
we see \expint{HomEI0} and \ETDint{SETD0} are now equally as
efficient. When the nondiagonal part dominates the diagonal part of the
noise then \figref{fig:5extra} shows that \expint{HomEI0} is still the
most efficient closely followed by the semi-implicit
  Euler--Maruyama method.
%the situation is similar to \figref{fig:3cd} and
%\expint{SETD0} is slightly more efficient than \expint{HomEI0}, see
%\figref{fig:5extra} where $\beta=0.1$ and $\alpha=1$. 

\begin{figure}[ht]
\begin{center}
(a) \hspace{.49\textwidth} (b) \\
\includegraphics[width=0.49\textwidth]{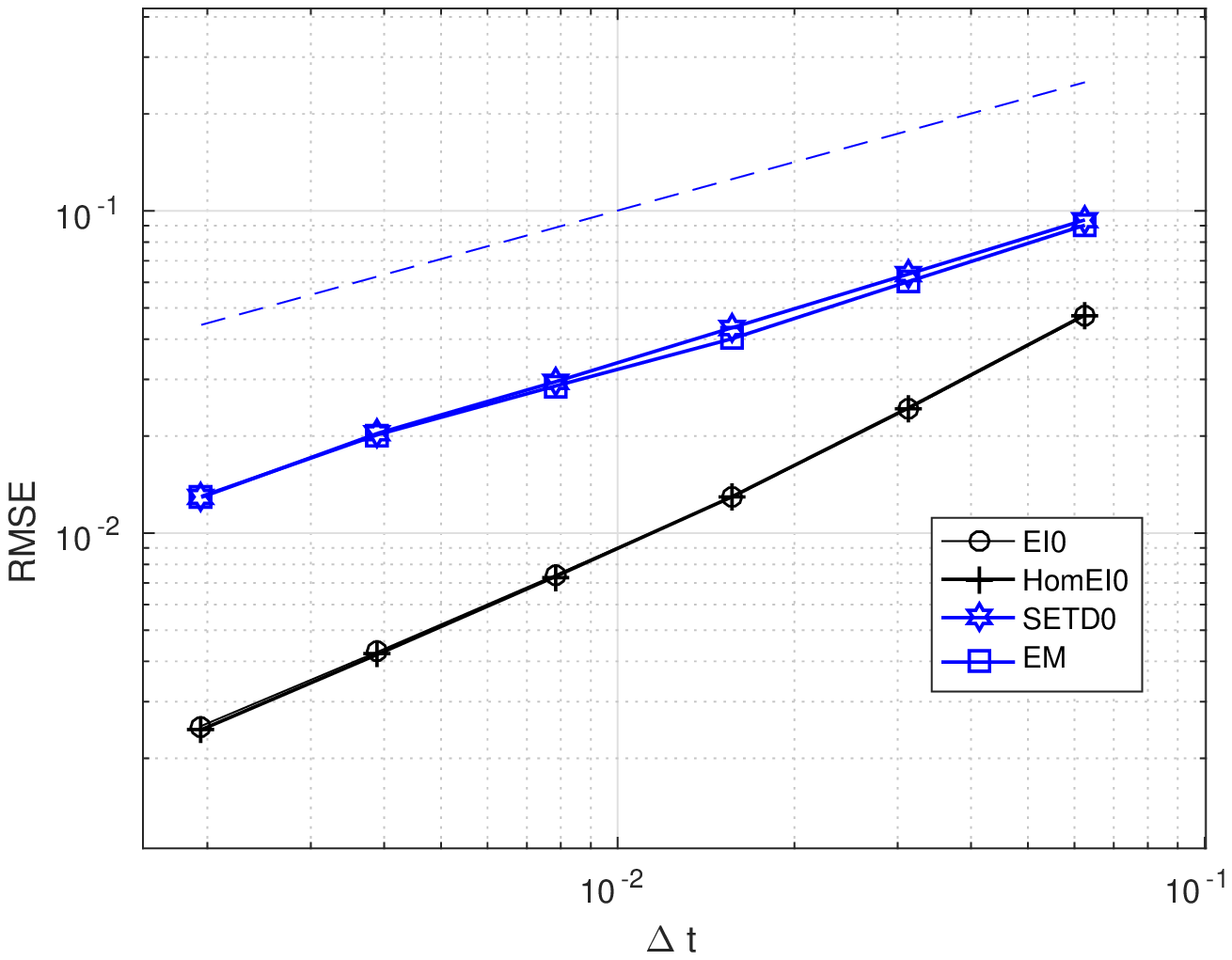}
\includegraphics[width=0.49\textwidth]{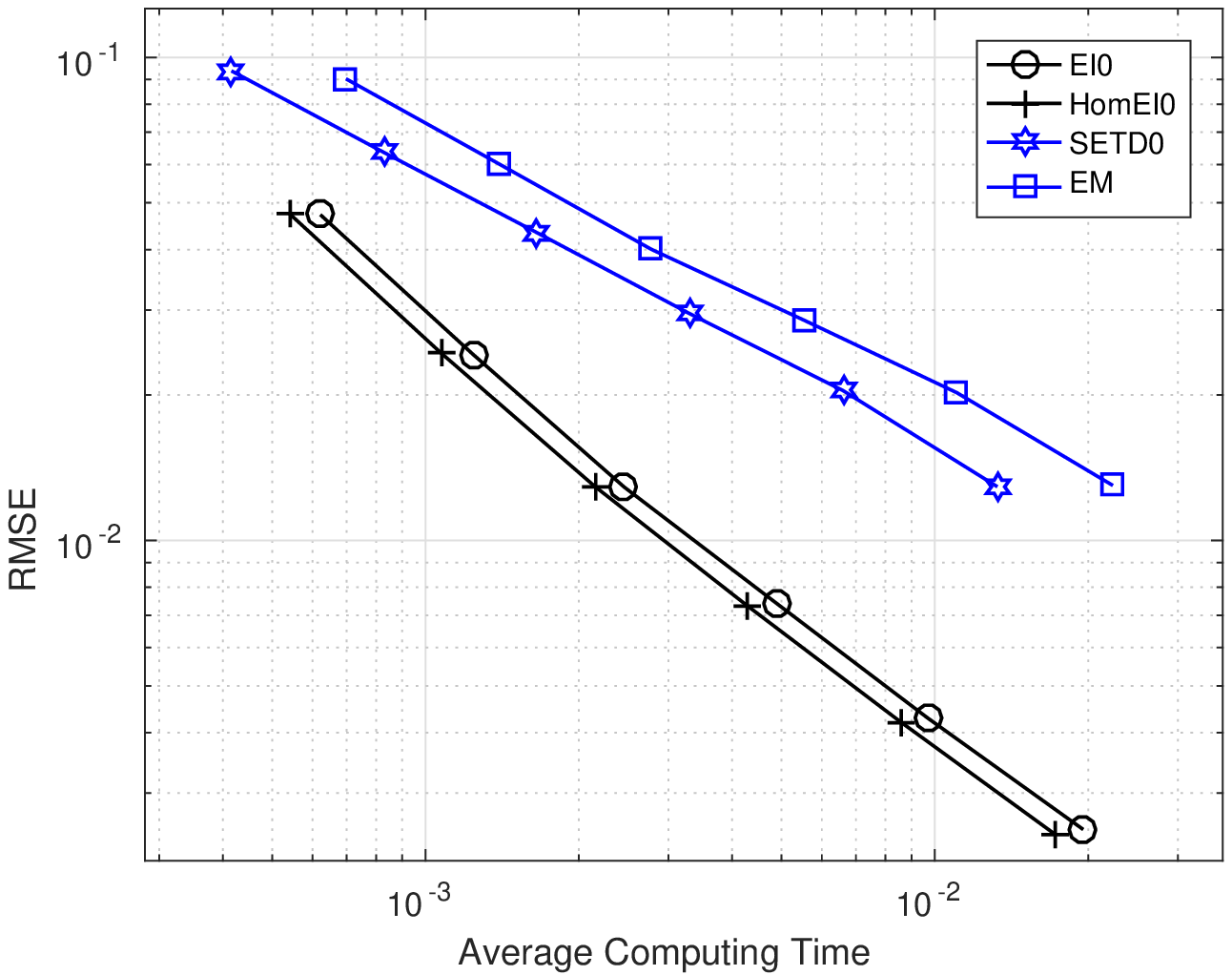}\\
(c) \hspace{.49\textwidth} (d) \\
\includegraphics[width=0.49\textwidth]{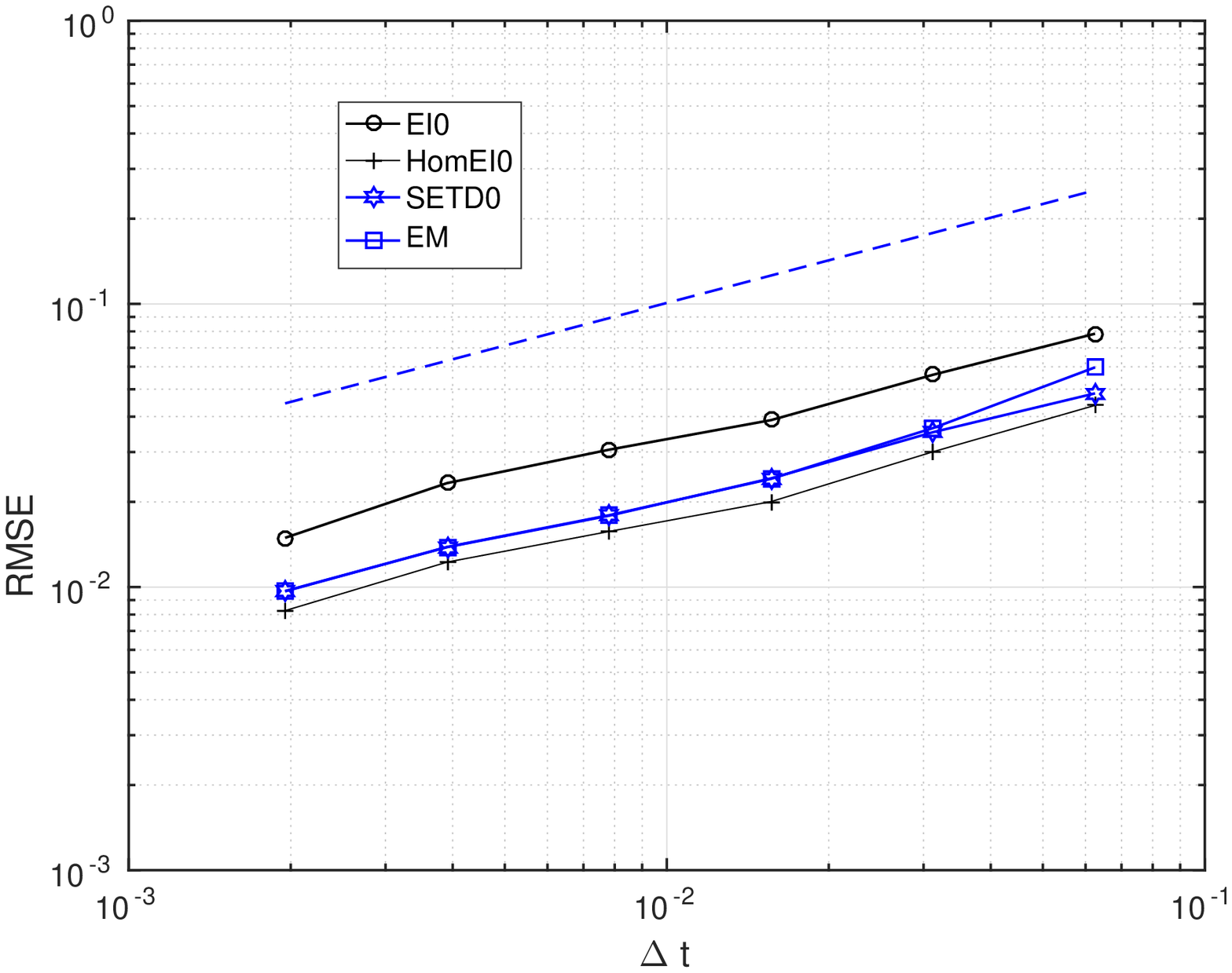}
\includegraphics[width=0.49\textwidth]{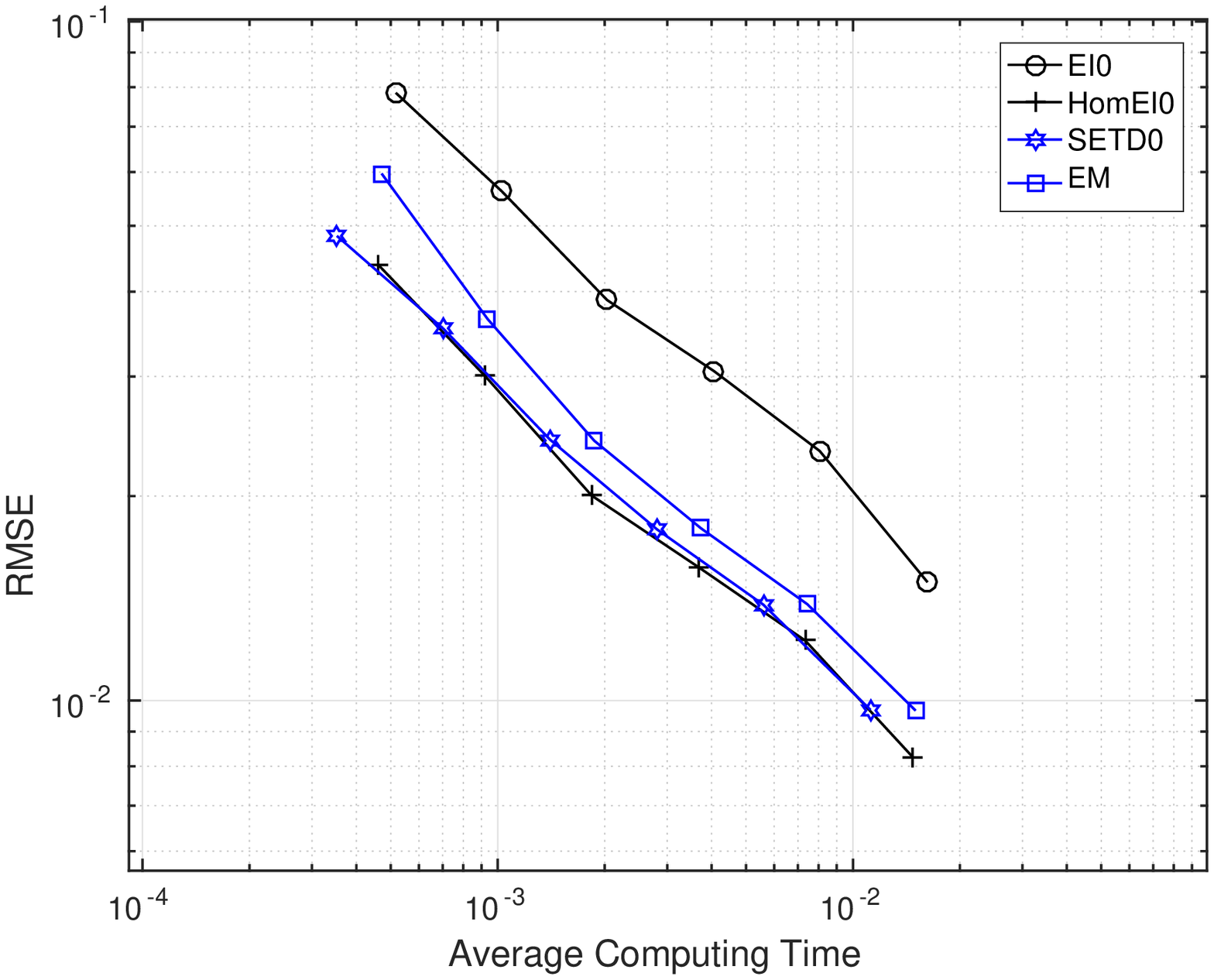}
\caption{Equation \eqref{eq:SPDEexample} with $r=4$, $T=1$ and
  $M=1000$ samples comparing in (a), (b) $\beta=1$, $\alpha=0.1$ 
  and in (c), (d)  $\beta=1$, $\alpha=1$.
  (a) and (c) show root mean square error against $\Dt$. Also plotted is a reference line
  with slope 1/2. (b) and (d)  root mean square error against
  cputime. Where the diagonal noise term dominates \expint{HomEI0} is
  the most efficient. For equal weighting we see \expint{HomEI0} is as
efficient as \ETDint{SETD0}.}
\label{fig:45}
\end{center}
\end{figure}

\begin{figure}[ht]
\begin{center}
(a) \hspace{.49\textwidth} (b) \\
\includegraphics[width=0.49\textwidth,height=4.5cm]{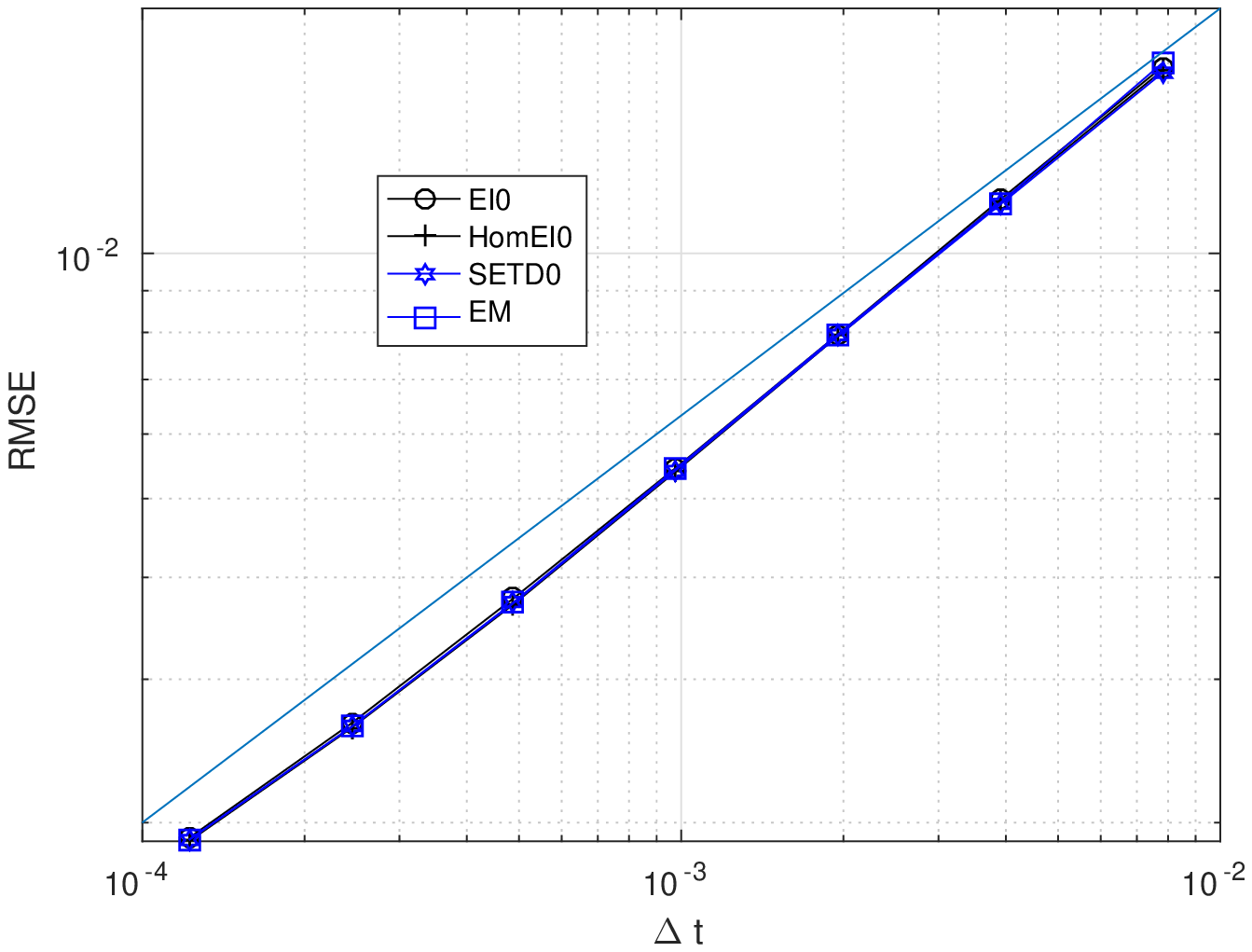}
\includegraphics[width=0.49\textwidth,height=4.5cm]{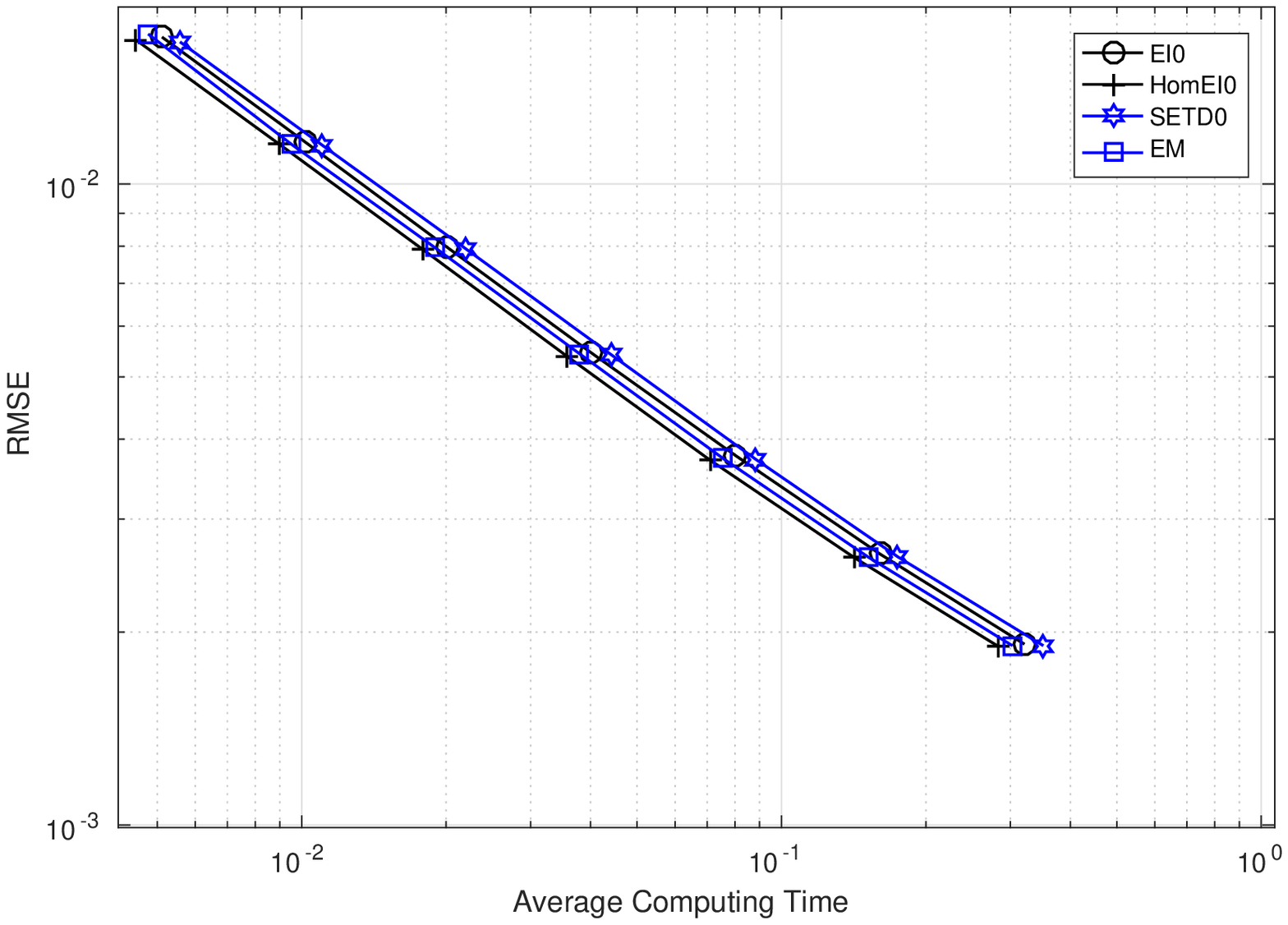}
\caption{Equation \eqref{eq:SPDEexample} with $r=4$, $T=1$ and
  $M=1000$ samples with $\beta=0.1$, $\alpha=1$.  (a) root
  mean square error against $\Dt$. Also plotted is a reference line
  with slope 1/2. (b)  root mean square error against cputime.
  The noise is dominated by the non diagonal term and 
  we now see that
  \expint{HomEI0} is the most efficient, followed by the semi-implicit
  Euler--Maruyama method.%\expint{ImpEM}.
  %We now see that
  %\expint{SETD0} is the most efficient, followed by \expint{HomEI0}.
}
\label{fig:5extra}
\end{center}
\end{figure}

\subsection*{Example 3 : Linear stiff SDE}
Finally we consider the following linear equation which is used as a
test equation for stiff solvers, see for example \cite{geometric} 
\begin{equation} \label{eq:stiff}
d \vec{u} (t)=\beta \begin{pmatrix}
0 & 1  \\
-1 & 0
\end{pmatrix}  \vec{u} (t)dt
+ 
\frac{\sigma}{2}  \begin{pmatrix}
1& 1  \\
1 & 1
\end{pmatrix} \vec{u} (t) d  W_1 (t) 
+
\frac{\rho}{2}     \begin{pmatrix}
1 & -1 \\
-1& 1 
\end{pmatrix} \vec{u} (t) d  W_2 (t)
\end{equation}
with initial condition $\vec{u}(0)=(1,0)^T$. The aim is to estimate
$\eval{\vec{u}(t)}$ for $t \in [0,T]$. It is known from theory that
solutions stay in the neighbourhood of the origin, see \cite{geometric}.
We perform simulations with $\beta=5$, $\sigma=4$, $\rho=0.5$ and
$T=50$ with a fixed time step of $\Dt=0.05$ and $M=1000$ realisations.
We compare approximations of $\eval{\vec{u}(t)}$ using \ETDint{SETD0},
and \expint{EI0}. To apply \expint{EI0}, we take 
\begin{equation}
  B_1= \frac{\sigma}{2}I_2, \quad  B_2=\frac{\rho}{2} I_2, \quad \vec{g}_1 (\vec{u})=\frac{\sigma}{2} \begin{pmatrix}
    u_2  \\
    u_1
  \end{pmatrix}, \quad \vec{g}_2 (\vec{u})=\frac{\rho}{2} \begin{pmatrix}
    -u_2  \\
    -u_1
  \end{pmatrix}.
\end{equation}
We observe in \figref{fig:6} that the \ETDint{SETD0} solution grows
rapidly away from the origin, for \expint{EI0} solutions are bounded
close to the origin and that the dynamics of  \expint{EI0} more
closely matches the dynamics of the underlying SDE. 
\begin{figure}[ht]
\begin{center}
(a) \hspace{.49\textwidth} (b) \\
\includegraphics[width=0.49\textwidth]{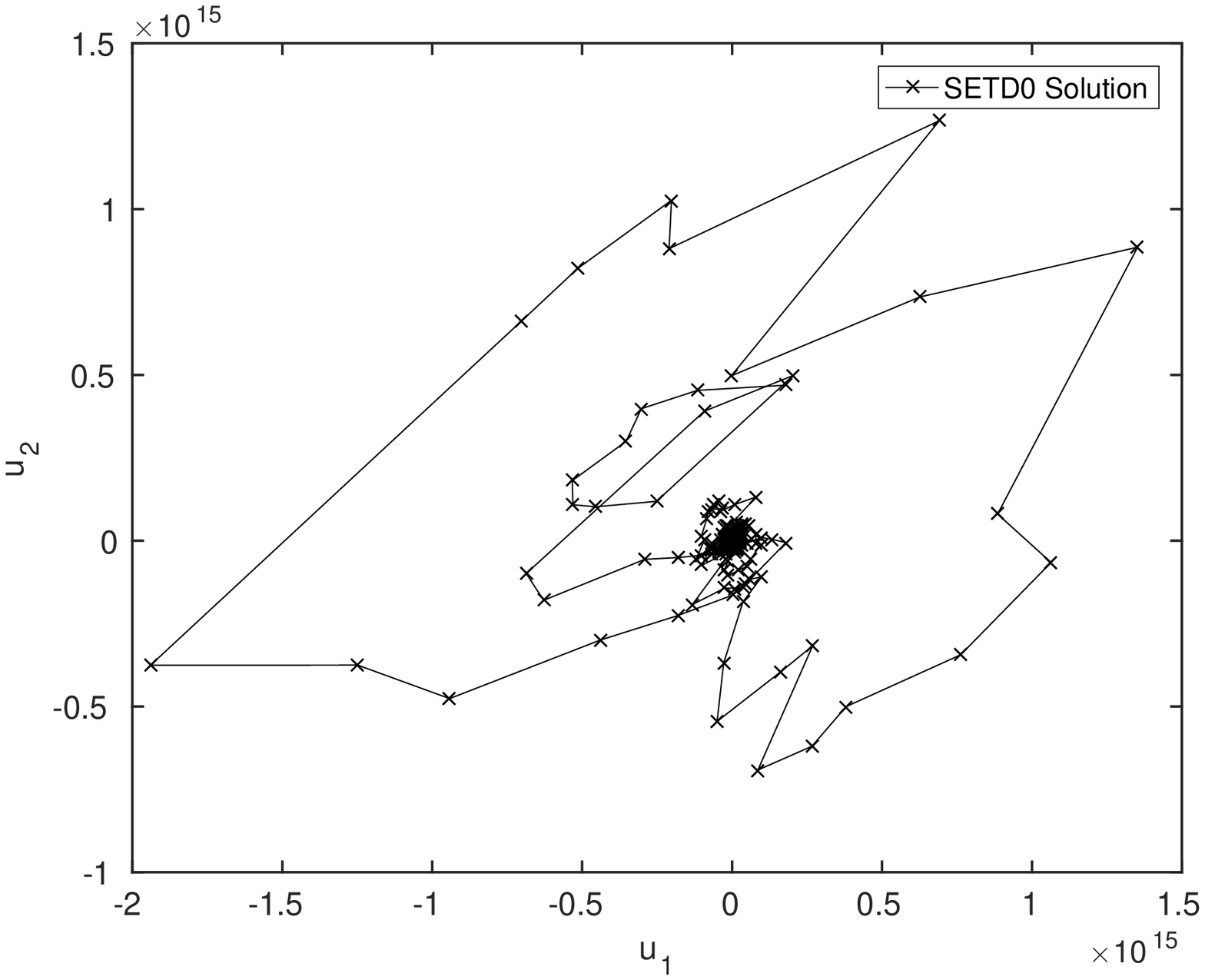}
\includegraphics[width=0.49\textwidth]{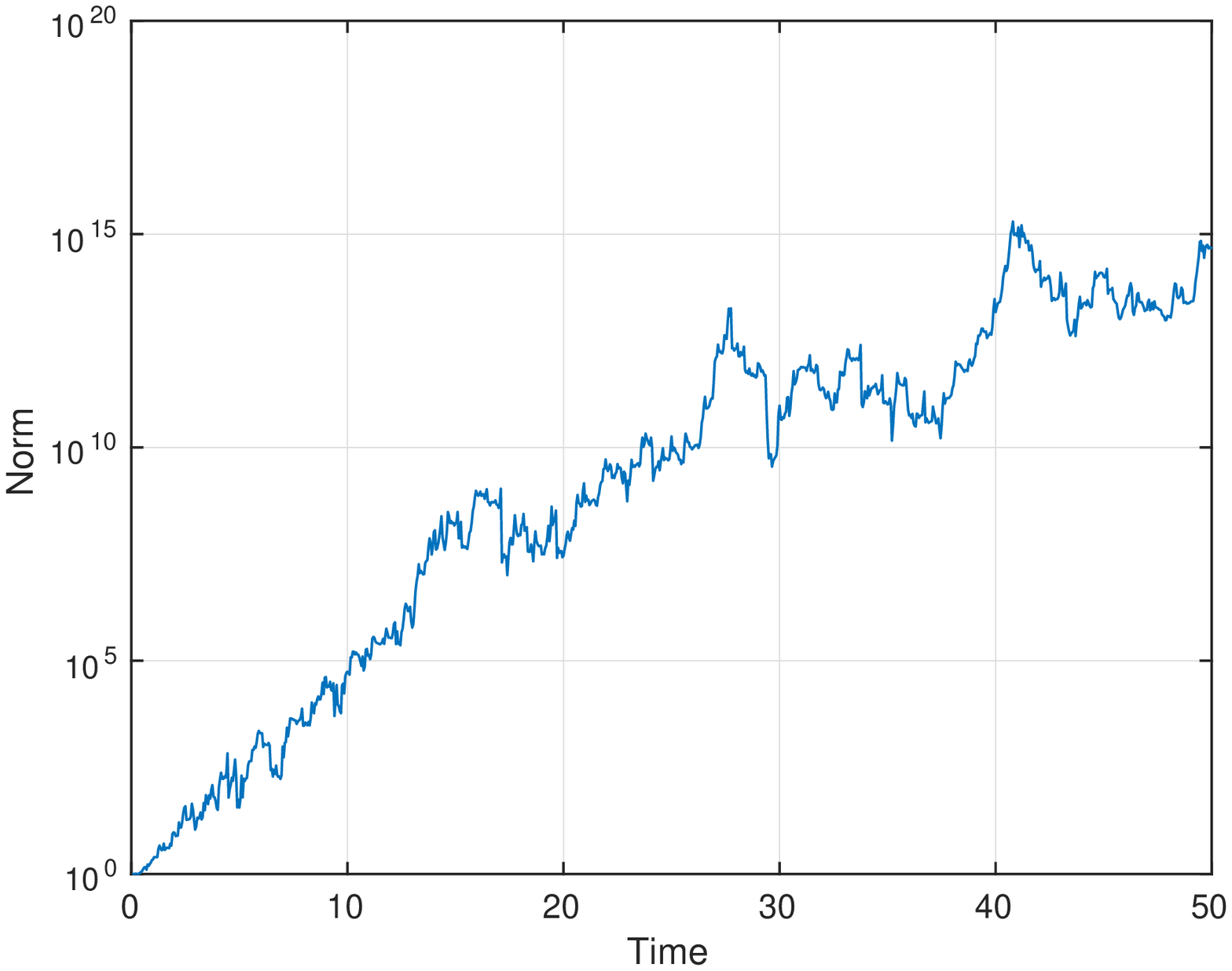}\\
(c) \hspace{.49\textwidth} (d) \\
\includegraphics[width=0.49\textwidth]{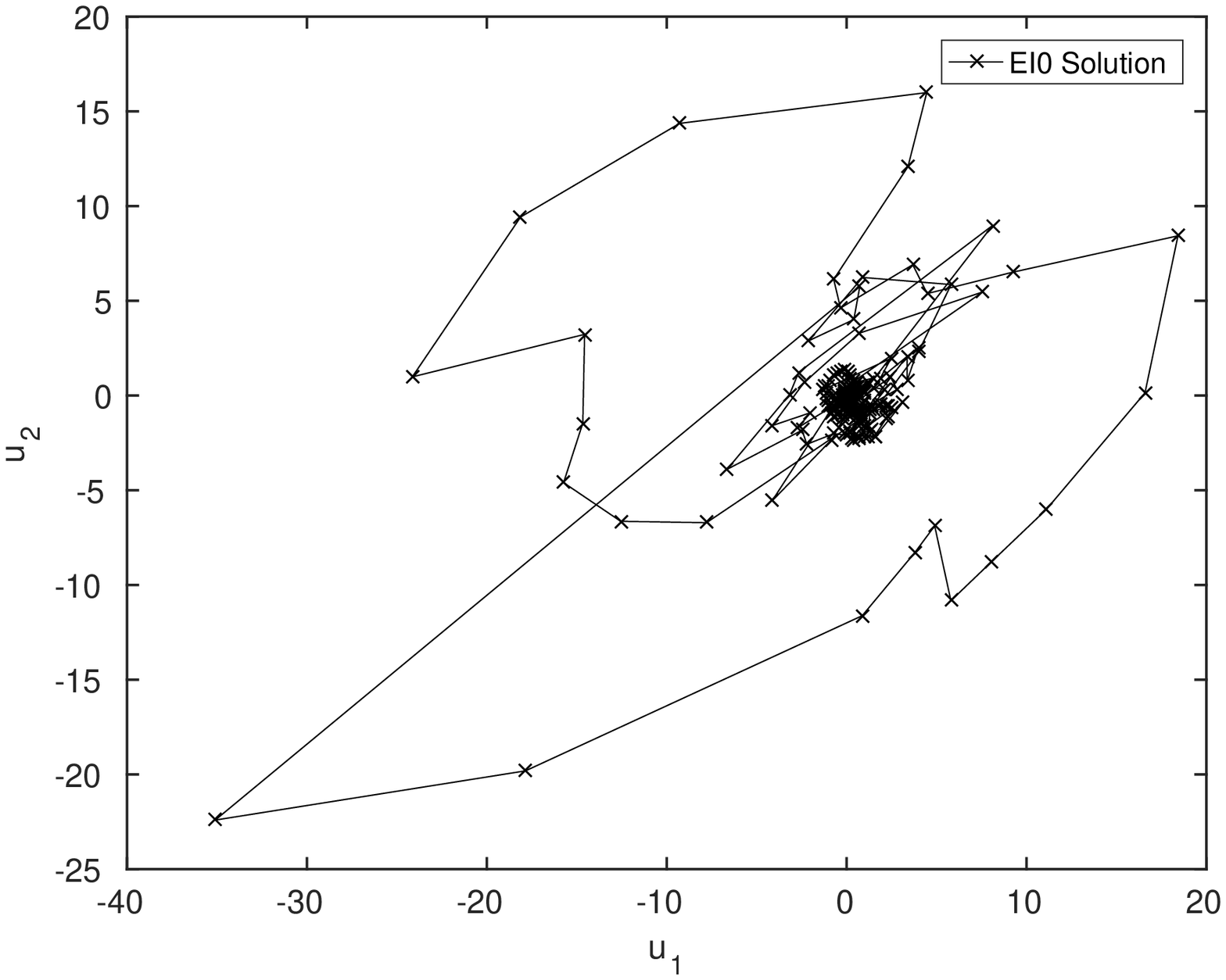}
\includegraphics[width=0.49\textwidth]{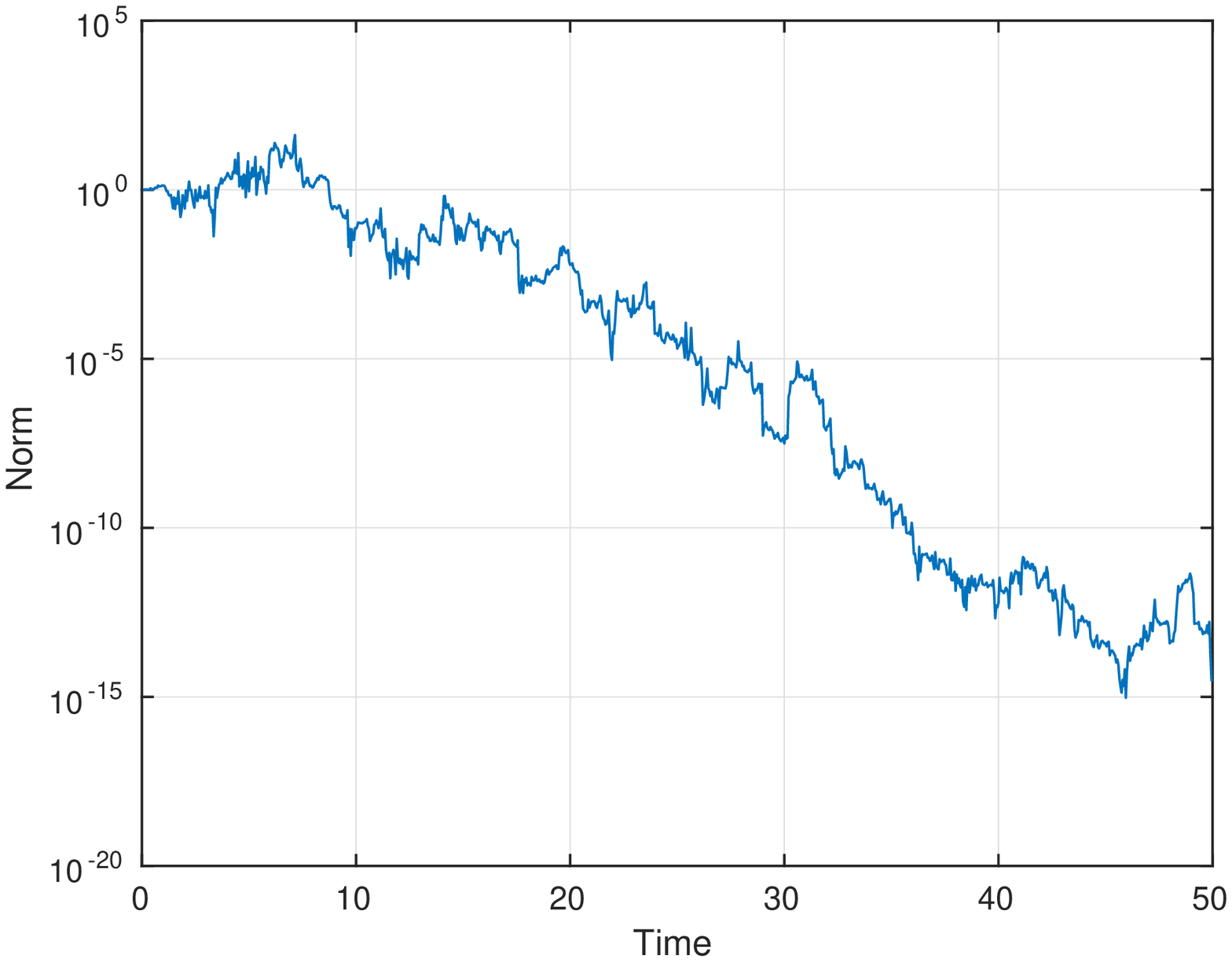}
\caption{Solution of \eqref{eq:stiff} with $T=50$,$\Delta t=0.05$, $M=1000$.
  In (a) we plot in the phase place the approximation to
  $\eval{\vec{u}(t)}$ found using \ETDint{SETD0} and in (b)
  $\|\eval{\vec{u}(t)}\|$.
  In (c) we plot in the phase place the approximation to
  $\eval{\vec{u}(t)}$ found using \expint{EI0} and in (d) $\|\eval{\vec{u}(t)}\|$.
  We see that \expint{EI0} better captures the true dynamics over this
  time interval.
}
\label{fig:6}
\end{center}
\end{figure}

\section{Proofs of the Main Results}
\label{sec:proofs}
Before giving the proof of main results, we need the following
results.%propositions. 

\begin{proposition}\label{Prop:1}
 Let Assumption \ref{ass:1}  hold. For each $T>0$ and
  $\vec{u} (0)=\vec{u}_0 \in  \mathbb{R} ^d$ there exists a unique 
  $\vec{u}$ satisfying \eqref{eq:EqABg} such that 
  $$
  \sup _{t  \in [0,T]} \Lnorm{\vec{u} (t)} =\sup _{t \in [0,T]}
  \eval{\norm{\vec{u} (t)} ^2} ^{1/2} < \infty.
  $$
  Furthermore, there exists  $K>0$ such that for $0\leq s,t \leq T$  
  \begin{equation}
    \Lnorm{\vec{u}(t)-\vec{u}(s)}\leq K \vert t-s  \vert ^{1/2}.
  \end{equation}
\end{proposition}
See \cite{gabrielbook} for the proof.

We now examine the remainder terms that arise from the local error.
Let us define the map for the exact flow 
\begin{multline} \label{eq:EI0_exact}
  \Psi_{\Exact}(t_{k+1},t_k,\vec{u}(t_k))= \\
\SG{t_{k+1},t_k} \vec{u}(t_k)+\SG{t_{k+1},t_k} \int _{t_k} ^{t_{k+1}} \SG{s,t_{k}}^{-1}
\tilde{\vec{f}} (\vec{u} (s))ds + \sum_{i=1} ^ m  \SG{t_{k+1},t_k}  \int _{t_k} ^{t_{k+1}}
  \SG{s,t_{k}}^{-1} \vec{g}_i(\vec{u}(s)) dW_i(s). 
\end{multline}
This exact flow will be used in analysis of \expint{EI0}. However, it
is more convenient to use the following Ito-Taylor expansion 
(see \eqref{eq:expansion}) to analyse \expint{\Mil}
\begin{multline} \label{eq:MilEI0_exact}
\Psi_{\Exact}(t_{k+1},t_k,\vec{u}(t_k))= \\ \SG{t_{k+1},t_k}
\vec{u}(t_k)+  \SG{t_{k+1},t_k}\int _{t_k} ^{t_{k+1}} \SG{s,t_{k}}^{-1} \tilde{\vec{f}}
(\vec{u} (s))ds  + \sum_{i=1}^m \SG{t_{k+1},t_k} \int _{t_k} ^{t_{k+1}} 
\vec{g}_i (\vec{u}(t_k)) dW_i(s) \\
+ \sum_{i=1}^m  \sum_{l=1}^m \SG{t_{k+1},t_k} \int
_{t_k} ^{t_{k+1}}\int _{t_k} ^s  \SG{r,t_{k}}^{-1}
\vec{H}_{i,l}(\vec{u}(r)) dW_l(r)dW_i(s)\\ + \sum_{i=1}^m \SG{t_{k+1},t_k} \int _{t_k}
^{t_{k+1}} \int _{t_k} ^s  \SG{r,t_{k}}^{-1} \vec{Q}_i(\vec{u}(r)) dr
dW_i(s).
\end{multline}
The numerical flows for  \expint{EI0} and \expint{\Mil} are given by
\begin{multline}\label{eq:flow_EI0}
  \Psi_{EI0}(t_{k+1},t_k,\vec{u}(t_k))=\SG{t_{k+1},t_k} \vec{u}(t_k)
  + \SG{t_{k+1},t_k} \int _{t_k} ^{t_{k+1}}  \tilde{\vec{f}} (\vec{u}
  (t_k))ds \\
  + \sum_{i=1}^m \SG{t_{k+1},t_k} \int _{t_k} ^{t_{k+1}}  \vec{g}_i(\vec{u}(t_k) )dW_i(s) 
\end{multline}
and
\begin{multline} \label{eq:flow_MilEI0}
  \Psi_{\MIL}(t_{k+1},t_k,\vec{u}(t_k))=\SG{t_{k+1},t_k}\vec{u}(t_k)+
  \SG{t_{k+1},t_k} \int _{t_k} ^{t_{k+1}}  \tilde{\vec{f}} (\vec{u}
  (t_k))ds \\
  + \sum_{i=1}^m \SG{t_{k+1},t_k} \int _{t_k} ^{t_{k+1}} 
  \vec{g}_i (\vec{u}(t_k)) dW_i(s)  
  + \sum_{i=1}^m  \sum_{l=1}^m  \SG{t_{k+1},t_k} \int
  _{t_k} ^{t_{k+1}}\int _{t_k} ^s 
  \vec{H}_{i,l}(\vec{u}(t_k)) dW_l(r)dW_i(s).
\end{multline}
First we look at the local error $R_{EI0}$ for \expint{EI0}, where 
$R_{EI0}$  is defined as
\begin{equation}
  \label{eq:Rdef1}
  R_{EI0}(t,s,\vec{u}(s))=\Psi_{\Exact}(t,s,\vec{u}(s))-\Psi_{EI0}(t,s,\vec{u}(s)).
\end{equation}

\begin{lemma}\label{lemma:1}
  Let  the Assumptions \ref{ass:1}  hold. Then
  \begin{equation}
    \Lnorm{\sum_{k=0} ^{N-1} \SG{t_N,t_{k+1}}   R_{EI0}(t_{k+1},t_k,\vec{u}(t_k))}^2=\mathcal{O}(\Delta t)
  \end{equation}
\end{lemma}
\begin{proof}
Considering the  exact flow \eqref{eq:EI0_exact} and the numerical flow \eqref{eq:flow_EI0}, the local error of \expint{EI0} is  given by
\begin{multline}
R_{EI0}(t_{k+1},t_k,\vec{u}(t_k))=\SG{t_{k+1},t_k}  \int _{t_k} ^{t_{k+1}}\left( \SG{s,t_{k}}^{-1}  \tilde{\vec{f}} (\vec{u} (s))-  \tilde{\vec{f}} (\vec{u} (t_k)\right)   ds  \\
 +\sum_{i=1}^m  \SG{t_{k+1},t_k}  \int _{t_k} ^{t_{k+1}}\left(   \SG{s,t_{k}}^{-1}  \vec{g}_i(\vec{u}(s))  -  \vec{g}_i(\vec{u}(t_k))\right)  dW_i(s).
\end{multline}
Adding and subtracting the terms   $\SG{s,t_{k}}^{-1}\tilde{\vec{f}}
(\vec{u} (t_k)$,  $\SG{s,t_{k}}^{-1} \vec{g}_i(\vec{u}(t_k))$ in the  first
and second integrals we have 
\begin{align*}
\lefteqn{ \Lnorm{\sum_{k=0} ^{N-1} \SG{t_N,t_{k+1}}  R_{EI0}(t_{k+1},t_k,\vec{u}(t_k))}^2
  \leq }&\\
& +4  \Lnorm{ \sum_{k=0} ^{N-1} \SG{t_N,t_{k}} \int _{t_k} ^{t_{k+1}} \SG{s,t_{k}}^{-1}\left(  \tilde{\vec{f}} (\vec{u} (s))- \tilde{\vec{f}} (\vec{u} (t_k)\right)   ds}^2 \\ 
& 4  \Lnorm{ \sum_{k=0} ^{N-1} \SG{t_N,t_{k}} \int _{t_k} ^{t_{k+1}} \left( \SG{s,t_k}^{-1} \tilde{\vec{f}} (\vec{u} (t_k))- \tilde{\vec{f}} (\vec{u} (t_k)\right)   ds}^2\\
& +4 \Lnorm{ \sum_{k=0} ^{N-1} \SG{t_N,t_{k}} \sum_{i=1}^m  \int _{t_k} ^{t_{k+1}} \SG{s,t_{k}}^{-1}\left(     \vec{g}_{i}(\vec{u}(s))  -   \vec{g}_{i}(\vec{u}(t_k))\right)  dW_i(s)  }^2 \\
& +4  \Lnorm{ \sum_{k=0} ^{N-1} \SG{t_N,t_{k}} \sum_{i=1}^m    \int _{t_k} ^{t_{k+1}}  \left(   \SG{s,t_{k}}^{-1}  \vec{g}_{i}(\vec{u}(t_k))  -  \vec{g}_{i}(\vec{u}(t_k))\right)  dW_i(s)  }^2 \\
& =I+II+III+IV.
\end{align*}
We now consider each of the terms $I$, $II$, $III$, $IV$ separately
and we start with $I$
\begin{align*}
I& \leq 4 N   \sum_{k=0} ^{N-1} \Lnorm{   \SG{t_{N},t_{k}}  \int _{t_k} ^{t_{k+1}}
  \SG{s,t_{k}}^{-1} \left(  \tilde{\vec{f}} (\vec{u} (s))-
   \tilde{\vec{f}} (\vec{u} (t_k)\right)   ds   }^2\\
 & \leq 4 N   \sum_{k=0} ^{N-1} C_k \Lnorm{ \int _{t_k} ^{t_{k+1}}
   \SG{s,t_{k}}^{-1} \left(  \tilde{\vec{f}} (\vec{u} (s))-
   \tilde{\vec{f}} (\vec{u} (t_k)\right)   ds   }^2 \\   
 & \leq 4 N  C \sum_{k=0} ^{N-1} \eval{\norm{ \int _{t_k} ^{t_{k+1}}
   \left( \SG{s,t_{k}}^{-1} \tilde{\vec{f}} (\vec{u} (s))-
   \tilde{\vec{f}} (\vec{u} (t_k)\right)   ds   }^2}
   \end{align*}
where $C= \sup_{k=0,1,2,...N-1} C_k $ and $C_k$ is due to boundedness of  $\SG{t_{N},t_{k}}$ in $L^2(\Omega, \mathbb{R}^d)$. However, in the following lines C is used as a generic constant  which may vary from line to line due to boundedness of $\SG{}$ and $\SG{}^{-1}$.  Now, Jensen's inequality, global
Lipschitz  property of  $\tilde{\vec{f}}$ and Proposition \ref{Prop:1}
are applied to get  
\begin{align*}
I   & \leq 4 N \Delta t C \sum_{k=0} ^{N-1}   \int _{t_k} ^{t_{k+1}}
  \eval{\norm{ \SG{s,t_{k}}^{-1} \left(  \tilde{\vec{f}} (\vec{u} (s))-
   \tilde{\vec{f}} (\vec{u} (t_k)\right)}^2}   ds   \\ 
   & \leq 4 T CL^2  \sum_{k=0} ^{N-1}   \int _{t_k} ^{t_{k+1}}
  \eval{\norm{ \left(   \vec{u} (s)-
    \vec{u} (t_k)\right)}^2}   ds   \\  
   & \leq 4 TC L^2  \sum_{k=0} ^{N-1}   \int _{t_k} ^{t_{k+1}}
    \vert s-t_k  \vert   ds=K_I \Delta t.
\end{align*}
Similarly, for $II$. It is  easy to see $II\leq K_{II} \Delta t$ by
considering the fact that $\eval{\norm{\left( \SG{s,t_k}^{-1} - I\right) \vec{v}}^2}\leq K
\vert s-t_k\vert$ for  any ${\mathcal{F} _ {t_k}  }$  measurable $\vec{v} \in L^2(\Omega,\mathbb{R}^d)$, which can be concluded from the Ito-Taylor expansion
of $\SG{s,t_k}^{-1}\vec{v}$. 
For the term $III$,
$$ 
III = 4 \Lnorm{  \SG{t_N,0}  \sum_{k=0} ^{N-1}  \sum_{i=1}^m  \int _{t_k} ^{t_{k+1}}\SG{s,0} ^{-1} \left(     \vec{g}_{i}(\vec{u}(s))  -   \vec{g}_{i}(\vec{u}(t_k))\right)  dW_i(s)  }^2 
$$
where $\SG{t_N,t_k}=\SG{t_N,0}\SG{t_k,0}^{-1}$  is due to commutativity of the matrices $A$ and $B_i$'s. We have by the Ito isometry

\begin{align*}
  III & \leq  4  C  \sum_{k=0} ^{N-1} \sum_{i=1}^m  \int _{t_k} ^{t_{k+1}} \eval{\norm{\SG{s,0} ^{-1} \left(     \vec{g}_{i}(\vec{u}(s))  -   \vec{g}_{i}(\vec{u}(t_k))\right)}^2}  ds \\
      & \leq  4 C  L^2 \sum_{k=0} ^{N-1}  K_{III} \int _{t_k} ^{t_{k+1}}  \vert s-t_k \vert  ds= \mathcal{O} (\Delta t )
\end{align*}
where  global Lipschitz property of $ \vec{g}_{i}$ and  Proposition \ref{Prop:1}  are used.
By a similar argument we  have $IV=\mathcal{O} (\Delta
t )$. Combining  $I$, $II$, $III$ and $IV$ we have the result.
%Therefore the overall sum of $I+II+III+IV$ is of only first order.
\end{proof}
We now prove \thref{thrm:1}.
By induction, we express the approximation of $\vec{u}(t_N)$ by
$\vec{u}_N$ found by \expint{EI0} at $t=t_N$ as
\begin{multline} \label{eq:sum_discrete_int}
\vec{u}_N= \SG{t_N,0} \vec{u} _0 
+ \sum_{k=0}^ {N-1}  \SG{t_N,t_k}  \int_{t_k} ^ {t_{k+1}}  \tilde{\vec{f}} (\vec{u} _k) ds +  \sum_{k=0}^ {N-1}  \sum_{i=1} ^m 
 \SG{t_N,t_k} \int_{t_k} ^ {t_{k+1}}   \vec{g}_i (\vec{u}_k)
dW_i(s). 
\end{multline}

Due to commutativity of the matrices $A$ and $B_i$'s, $\SG{t_N,t_k}=\SG{t_N,0}\SG{t_k,0}^{-1}$, the second matrix $\SG{t_k,0}^{-1}$ can be put inside the stochastic integrals as well as deterministic integral. 
Now we define  the continuous time process $\vec{u}_{\Delta t} (t)$
for \eqref{eq:sum_discrete_int} that agrees  with approximation $\vec{u}_k$
at  $t=t_k$. By  introducing the  variable $\hat{t}=t_k$  for $t_k\leq
t< t_{k+1}$, 
\begin{multline}
\vec{u} _{\Delta t} (t) = \SG{t,0} \vec{u} _{\Delta t} (0)+  \SG{t,0}\int_{0} ^ {t}  \SG{\hat{s},0}^{-1} \tilde{\vec{f}} (\vec{u} _{\Delta t} (\hat{s}))  ds +  \sum_{i=1} ^m  \SG{t,0}  \int_{0} ^{ t} \SG{\hat{s},0}^{-1}  \vec{g}_i(\vec{u} _{\Delta t} (\hat{s})) dW_i(s).
\end{multline}
This continuous version has the  property that $\vec{u} _{\Delta t}
(t_k) =\vec{u}_k$. By recalling definition of local error,
the iterated sum of the exact solution at $t=t_N$ 
is fonud by induction to be 
\begin{multline} \label{eq:EI0_exact_sum}
\vec{u}(t_N)= \SG{t_N,0}\vec{u}_0+  \SG{t_N,0}  \int _{0} ^{t_N} \SG{\hat{s},0}^{-1} \tilde{\vec{f}}
(\vec{u} (\hat{s}))ds \\ + \SG{t_N,0}  \sum_{i=1}^m  \int _{0} ^{t_N} 
\SG{\hat{s},0}^{-1}\vec{g}_i (\vec{u}(\hat{s})) dW_i(s)
 + \sum_{k=0}^ {N-1} \SG{t_N,t_{k+1}}   R_{EI0}(t_{k+1},t_k, \vec{u} (t_k)).
\end{multline}
Denoting the error by $\vec{e}(\hat{t})=\vec{u}(\hat{t})-\vec{u}
_{\Delta t} (\hat{t})$, 
we see that 
%\begin{multline}
$$\Lnorm{\vec{e}(\hat{t})}^2 \leq  
 \left( 3 \hat{t} C  L^2 +  3 C m L^2 \right) \int_{0} ^{ \hat{t}} \eval{\norm{\vec{e} (s)}^2}ds +3  K\Delta t,$$
%\end{multline}
%\begin{multline}
%\Lnorm{\vec{e}(\hat{t})}^2 \leq  \\
% 3 \hat{t} C  L^2 \int_{0} ^{ \hat{t}} \eval{\norm{\vec{e} (s)}^2}ds 
%+ 3 C m L^2 \int_{0} ^{ \hat{t}} \eval{\norm{\vec{e} (s)}^2}ds +3  K\Delta t
%\end{multline}
where $L$ is the largest one of the Lipschitz constants of the functions $\vec{g}_i$, $\tilde{\vec{f}}$. Finally,   Gronwall's inequality completes the proof.

\subsection{Proof of \thref{thrm:2}} 

We now examine the local error for \expint{\Mil}, given by \eqref{eq:MilEI0}.
\begin{lemma}\label{lemma:2} Let Assumptions \ref{ass:1} and
  \ref{ass:2} and \ref{ass:3} hold. Then 
  \begin{equation}
    \Lnorm{\sum_{k=0} ^{N-1} \SG{t_N,t_k+1}  R_{\MIL}(t_{k+1},t_k,\vec{u}(t_k))}^2=\mathcal{O}(\Delta t^2)
  \end{equation}
  where $R_{\MIL}$  is defined as
  \begin{equation*}
    \label{eq:Rdef}
    R_{\MIL}(t,s,\vec{u}(s))=\Psi_{\Exact}(t,s,\vec{u}(s))-\Psi_{\MIL}(t,s,\vec{u}(s)).
  \end{equation*}
\end{lemma}

\begin{proof}
Considering the  exact flow \eqref{eq:MilEI0_exact} and the numerical flow \eqref{eq:flow_MilEI0}  corresponding to the scheme \expint{\Mil}, we have 
\begin{multline}
   R_{\MIL}(t_{k+1},t_k,\vec{u}(t_k)) =   \SG{t_{k+1},t_k} \int _{t_k} ^{t_{k+1}} \left( \SG{s,t_{k}}^{-1} \tilde{\vec{f}} (\vec{u} (s))-  \tilde{\vec{f}} (\vec{u} (t_k)\right)   ds \\+   \sum_{i=1}^m  \sum_{l=1}^m \SG{t_{k+1},t_k}\int _{t_k} ^{t_{k+1}}\int _{t_k} ^s\left(   \SG{r,t_{k}}^{-1}  \vec{H}_{i,l}(\vec{u}(r))  -   \vec{H}_{i,l}(\vec{u}(t_k))\right)  dW_l(r)dW_i(s)  \\
    +  \sum_{i=1}^m \SG{t_{k+1},t_k} \int _{t_k} ^{t_{k+1}} \int _{t_k} ^s \SG{r,t_{k}}^{-1}  \vec{Q}_i(\vec{u}(r)) dr dW_i(s).
\end{multline}
Adding and subtracting the terms  $\SG{s,t_k} ^{-1} \tilde{\vec{f}}
(\vec{u} (t_k)$,  $\SG{s,t_k} ^{-1}\vec{H}_{i,l}(\vec{u}(t_k))$ in the
first and second integrals respectively and  summing and taking the
norm and applying Jensen's inequality, we have
$$ \Lnorm{\sum_{k=0} ^{N-1} \SG{t_N,t_{k+1}}  R_{\MIL}(t_{k+1},t_k,\vec{u}(t_k))}^2 \leq I+II+III+IV+V$$
with 
$$I:=5 \Lnorm{ \sum_{k=0} ^{N-1} \SG{t_{N},t_k} \int _{t_k} ^{t_{k+1}}  \SG{s,t_k} ^{-1} \left(
    \tilde{\vec{f}} (\vec{u} (s))- 
    \tilde{\vec{f}} (\vec{u} (t_k)\right)   ds}^2$$
$$II:= 5  \Lnorm{ \sum_{k=0} ^{N-1} \SG{t_{N},t_k} \int _{t_k} ^{t_{k+1}} \left(
    \SG{s,t_{k}}^{-1} \tilde{\vec{f}} (\vec{u} (t_k))- 
    \tilde{\vec{f}} (\vec{u} (t_k))\right)   ds}^2$$
$$III:=5 \Lnorm{ \sum_{k=0} ^{N-1} \sum_{i=1}^m  \sum_{l=1}^m \SG{t_{N},t_k} \int
  _{t_k} ^{t_{k+1}}\int _{t_k} ^s \SG{r,t_k} ^{-1} \left(   
    \vec{H}_{i,l}(\vec{u}(r))  - 
    \vec{H}_{i,l}(\vec{u}(t_k))\right)  dW_l(r)dW_i(s)  }^2$$    

$$IV:=5 \Lnorm{ \sum_{k=0} ^{N-1} \sum_{i=1}^m  \sum_{l=1}^m  \SG{t_{N},t_k} \int _{t_k} ^{t_{k+1}}\int _{t_k} ^s\left(   \SG{r,t_{k}}^{-1}  \vec{H}_{i,l}(\vec{u}(t_k))  -   \vec{H}_{i,l}(\vec{u}(t_k))\right)  dW_l(r)dW_i(s)  }^2$$

and remainder
$$V=:5 \Lnorm{\sum_{k=0} ^{N-1}  \sum_{i=1}^m \SG{t_N,t_k} \int _{t_k} ^{t_{k+1}}
  \int _{t_k} ^s  \SG{r,t_{k}}^{-1}  \vec{Q}_i(\vec{u}(r))dr  dW_i(s)}^2.$$ 
We now consider each of the terms $I$, $II$, $III$, $IV$, $V$ separately
and we start with $I$. By Assumption \ref{ass:3},  we  have the following Ito-Taylor expansion for $\tilde{\vec{f}}$ 
\begin{align*}
\tilde{\vec{f}} (\vec{u(s)})& =\tilde{\vec{f}} (\vec{u}(t_k))+
  \sum_{i=1}^{m}D\tilde{\vec{f}} (\vec{u} (t_k)) \left( B_i
  \vec{u}(t_k)+\vec{g}_i( \vec{u}(t_k)) \right)
  \left(W_i(s)-W_i(t_k) \right)+R_{\tilde{f}} \\
& = \tilde{\vec{f}} (\vec{u}(t_k))+ \sum_{i=1}^{m}K_i\left(W_i(s)-W_i(t_k) \right)+R_{\tilde{f}}.
\end{align*}
We know that $R_{\tilde{f}}=\mathcal{O} (s-t_k)$, see for example \cite{gabrielbook}. 
By Jensen's inequality for the sum and  Ito-Taylor expansion,
\begin{align*}
I  \leq & 10 \eval{ \norm{  \sum_{k=0} ^{N-1}  \SG{t_N,t_k}  \int _{t_k} ^{t_{k+1}}
   \left(  \sum_{i=1}^{m} K_i \left(W_i(s)-W_i(t_k) \right)  \right)   ds   }^2} \\
& +10  \eval{ \norm{  \sum_{k=0} ^{N-1} \SG{t_N,t_k}   \int _{t_k} ^{t_{k+1}}
   R_{\tilde{f}}  ds   }^2}.
\end{align*}
By boundedness of  $\SG{t_N,t_k}$, we have
$$
I \leq 10 C \eval{ \norm{  \sum_{k=0} ^{N-1}   \int _{t_k} ^{t_{k+1}}
   \left(  \sum_{i=1}^{m} K_i \left(W_i(s)-W_i(t_k) \right)  \right)   ds   }^2}
 +10 C \eval{ \norm{  \sum_{k=0} ^{N-1}   \int _{t_k} ^{t_{k+1}}
   R_{\tilde{f}}  ds   }^2}.
$$
%where $C=\sup_{k=0,1,\ldots,N-1}C_k$ and 
Let us write $I=I_a+I_b$ and investigate the first term $I_a$. By the orthogonality relation
$\eval{<\Theta_k,\Theta_l>}=0$, $k\neq l$  for $$\Theta _k= \int _{t_k} ^{t_{k+1}}
   \left(  \sum_{i=1}^{m} K_i  \left(W_i(s)-W_i(t_k) \right)  \right)
   ds,$$ 
we have
$$
% \eval{ \norm{ \sum_{k=0} ^{N-1}   \int _{t_k} ^{t_{k+1}}
%   \left(  \sum_{i=1}^{m} K_i \left(W_i(s)-W_i(t_k) \right)  \right)
%   ds   }^2}
I_a= \sum_{k=0} ^{N-1}   \eval{ \norm{   \int _{t_k} ^{t_{k+1}}
   \left(  \sum_{i=1}^{m} K_i \left(W_i(s)-W_i(t_k) \right)  \right)   ds   }^2}.
$$    
By two applications of Jensen's inequality for the integral and sum
$$
I_a \leq  K \Delta t \sum_{k=0} ^{N-1} \sum_{i=1} ^{m}  \int _{t_k} ^{t_{k+1}} \eval{ \norm{  W_i(s)-W_i(t_k) }^2 }ds=\mathcal{O} (\Delta t^2).
$$
Since $R_f$ contains higher order terms, we conclude $I=\mathcal{O} (\Delta t^2)$. Similarly, for II we find the same order by following same arguments. 
 
For III, we have by Ito isometry applied consecutively for outer and
inner stochastic integrals 
\begin{align*}
\lefteqn{III\leq}&  \\
& 5  C    \Lnorm{  \sum_{k=0} ^{N-1}   \sum_{i=1}^m  \sum_{l=1}^m  \int _{t_k} ^{t_{k+1}}\int _{t_k} ^s \SG{r,0}^{-1}\left(    \vec{H}_{i,l}(\vec{u}(r))  -   \vec{H}_{i,l}(\vec{u}(t_k))\right)  dW_l(r)dW_i(s)  }^2 \\ 
& = 5 C    \sum_{k=0} ^{N-1} \sum_{i=1}^m   \int _{t_k} ^{t_{k+1}} \eval{\norm{  \left(  \sum_{l=1}^m \int _{t_k} ^s  \SG{r,0}^{-1}\left(    \vec{H}_{i,l}(\vec{u}(r))  - \vec{H}_{i,l}(\vec{u}(t_k))\right)  dW_l(r)\right) }^2}ds \\
& = 5  C    \sum_{k=0} ^{N-1}   \sum_{i=1}^m  \sum_{l=1}^m  \int _{t_k} ^{t_{k+1}} \int
_{t_k} ^s   \eval{\norm{ \SG{r,0}^{-1}\left(   
    \vec{H}_{i,l}(\vec{u}(r))  -
    \vec{H}_{i,l}(\vec{u}(t_k)) \right)}^2}  drds.
\end{align*}
By global Lipschitz property of $ \vec{H}_{i,l}$ and  Proposition \ref{Prop:1} 
\begin{align*}
III
& \leq K_{III}  \sum_{k=0} ^{N-1} \int _{t_k} ^{t_{k+1}} \int _{t_k}
^s \vert r-t_k \vert  dr ds= K_{III} \Delta t ^2.
\end{align*}
In a similar way,  it can be shown that $IV \leq  K_{IV} \Delta t ^2$.
Since $\int _{t_k} ^{t_{k+1}}\int_{t_k}^s drdW(s) \sim N (0,
\frac{1}{3} \Delta t ^3 )$, it is straightforward to
see $ V \leq K_{V} \Delta t ^2$.   

%The over all sum is $\mathcal{O}(\Delta t^2)$.
\end{proof}

We now prove \thref{thrm:2}. As in the proof of  \thref{thrm:1},
we define  the continuous time process $\vec{u}_{\Delta t} (t)$
for \expint{\Mil} that agrees  with approximation $\vec{u}_k$
at  $t=t_k$. By  introducing the  variable $\hat{t}=t_k$  for $t_k\leq
t< t_{k+1}$, 
\begin{multline}
\vec{u} _{\Delta t} (t) = \SG{t,0} \vec{u} _{\Delta t} (0)+  \SG{t,0}\int_{0} ^ {t}  \SG{\hat{s},0}^{-1} \tilde{\vec{f}} (\vec{u} _{\Delta t} (\hat{s}))  ds +  \sum_{i=1} ^m  \SG{t,0}  \int_{0} ^{ t} \SG{\hat{s},0}^{-1}  \vec{g}_i(\vec{u} _{\Delta t} (\hat{s})) dW_i(s) \\
+ \sum_{i=1} ^m   \sum_{l=1} ^m   \SG{t,0} \int_{0} ^ {t} \int_{\hat{s}} ^ {s} \SG{\hat{s},0}^{-1} \vec{H}_{i,l}(\vec{u} _{\Delta t} (\hat{s}))  dW_l(r)dW_i(s).
\end{multline}
The iterated sum of the exact solution  at $t=t_N$ is obtained
inductively to be
\begin{multline} \label{eq:MilEI0_exact_sum}
\vec{u}(t_N)= \SG{t_N,0}\vec{u}_0+  \SG{t_N,0}  \int _{0} ^{t_N} \SG{\hat{s},0}^{-1} \tilde{\vec{f}}
(\vec{u} (\hat{s}))ds + \SG{t_N,0}  \sum_{i=1}^m  \int _{0} ^{t_N} 
\SG{\hat{s},0}^{-1}\vec{g}_i (\vec{u}(\hat{s})) dW_i(s) 
\\+ \sum_{i=1}^m  \sum_{l=1}^m \SG{t_N,0} \int
_{0} ^{t_N} \int _{\hat{s}} ^s  \SG{\hat{s},0}^{-1}
\vec{H}_{i,l}(\vec{u}(\hat{s})) dW_l(r)dW_i(s)+ \sum_{k=0}^ {N-1} \SG{t_N,t_{k+1}}   R_{\MIL}(t_{k+1},t_k, \vec{u} (t_k))
\end{multline}
Denoting the error by $\vec{e}(\hat{t})=\vec{u}(\hat{t})-\vec{u}
_{\Delta t} (\hat{t})$, 
we see that 
\begin{multline}
\Lnorm{\vec{e}(\hat{t})}^2 \leq 
 \left( 4 \hat{t} C  L^2 + 4 CL^2m(1+m) 
\right)\int_{0} ^{ \hat{t}} \eval{\norm{\vec{e} (s)}^2}ds 
+4  K\Delta t ^2.
\end{multline}
%\begin{multline}
%\Lnorm{\vec{e}(\hat{t})}^2 \leq  \\
% 4 \hat{t} C  L^2 \int_{0} ^{ \hat{t}} \eval{\norm{\vec{e} (s)}^2}ds 
%+ 4 C m L^2 \int_{0} ^{ \hat{t}} \eval{\norm{\vec{e} (s)}^2}ds 
%+ 4 C m^2 L^2 \int_{0} ^{ \hat{t}} \eval{\norm{\vec{e} (s)}^2}ds+4  K\Delta t ^%2
%\end{multline}
Finally,   Gronwall's inequality completes the proof.%reveals the order of convergence.

\section{Conclusion and Remarks}
Exponential integrators that take  advantage of
Geometric Brownian Motion have been derived and their strong
convergence properties discussed. Furthermore we introduced
a homotopy based scheme that can take advantage of linearity in the
diffusion and also effectively handle nonlinear noise.
The proposed schemes are particularly well suited to the SDEs
arising from the semi-discretisation of a SPDE where typically
diagonal noise arises. Where the SDEs are not of the semi-linear form
of \eqref{eq:EqABg} then a Rosenbrock type method could be applied,
similar to \cite{doi:10.1137/080717717}.
% and the extension of convergence analysis of these
%exponential integrators for SPDEs is in preparation. For 
As mentioned in \secref{sec:Mainresults} the exponential integrators
suggest new forms of taming coefficients for for SDEs with non
globally Lipschitz drift and diffusion terms \cite{hutzenthaler2015numerical},
see \cite{ErdoganLordTaming}.
Our numerical examples show that these new exponential based schemes 
are more efficient than the standard integrators and also deal well
with the stiff linear problem.
In addition we see the effectiveness of the homotopy
approach with the simple choice of parameter in \eqref{eq:hompar},
(it would be interesting to investigate an adaptive choice in the
future).

%Another  interesting implication of  exponential integrators is that
%they suggest new kinds of taming coefficients for SDEs with non
%globally Lipshitz drift and diffusion terms
%\citep{hutzenthaler2015numerical}. 

%
\section*{Acknowledgements}
\noindent
The first author was supported by Tubitak grant : 2219 International
Post-Doctoral Research Fellowship Programme and work was
completed at the Department of Mathematics, Maxwell Institute, Heriot
Watt University, UK.
We would like to thank Raphael Kruse for his comments on an earlier draft.

%  
%    \leq K (\Delta t)
\bibliographystyle{plain}
\bibliography{mysource}

\begin{thebibliography}{10}

\bibitem{BeckerJentzenKloeden}
S.~Becker, A.~Jentzen, and P.~Kloeden.
\newblock An exponential {Wagner}--{Platen} type scheme for spdes.
\newblock {\em SIAM J. Numer. Anal.}, 2016.

\bibitem{biscay1996local}
R~Biscay, JC~Jimenez, JJ~Riera, and PA~Valdes.
\newblock Local linearization method for the numerical solution of stochastic
  differential equations.
\newblock {\em Annals of the Institute of Statistical Mathematics},
  48(4):631--644, 1996.

\bibitem{CarbonellJimenez}
F.~Carbonell and J.~C. Jimenez.
\newblock Weak local linear discretizations for stochastic differential
  equations with jumps.
\newblock {\em J. Appl. Probab.}, 45(1):201--210, 2008.

\bibitem{ErdoganLordTaming}
Utku Erdogan and Gabriel~J Lord.
\newblock Tamed exponential integrators for stochastic differenatial equations.
\newblock In Preparation.

\bibitem{hochbruck2010exponential}
Marlis Hochbruck and Alexander Ostermann.
\newblock Exponential integrators.
\newblock {\em Acta Numerica}, 19:209--286, 2010.

\bibitem{doi:10.1137/080717717}
Marlis Hochbruck, Alexander Ostermann, and Julia Schweitzer.
\newblock Exponential rosenbrock-type methods.
\newblock {\em SIAM Journal on Numerical Analysis}, 47(1):786--803, 2009.

\bibitem{hutzenthaler2015numerical}
Martin Hutzenthaler and Arnulf Jentzen.
\newblock {\em Numerical approximations of stochastic differential equations
  with non-globally {Lipschitz} continuous coefficients}, volume 236.
\newblock American Mathematical Society, 2015.

\bibitem{JentzenPW}
Arnulf Jentzen.
\newblock Pathwise numerical approximation of {SPDE}s with additive noise under
  non-global {L}ipschitz coefficients.
\newblock {\em Potential Anal.}, 31(4):375--404, 2009.

\bibitem{Jentzen}
Arnulf Jentzen.
\newblock Taylor expansions of solutions of stochastic partial differential
  equations.
\newblock {\em Discrete Contin. Dyn. Syst. Ser. B}, 14(2):515--557, 2010.

\bibitem{JentzenKloedenOrder}
Arnulf Jentzen and Peter~E. Kloeden.
\newblock Overcoming the order barrier in the numerical approximation of
  stochastic partial differential equations with additive space-time noise.
\newblock {\em Proc. R. Soc. Lond. Ser. A Math. Phys. Eng. Sci.},
  465(2102):649--667, 2009.

\bibitem{expmil}
Arnulf Jentzen and Michael R{\"o}ckner.
\newblock A {M}ilstein scheme for {SPDEs}.
\newblock {\em Foundations of Computational Mathematics}, 15(2):313--362, 2015.

\bibitem{JimenezCarbonell}
J.~C. Jimenez and F.~Carbonell.
\newblock Convergence rate of weak local linearization schemes for stochastic
  differential equations with additive noise.
\newblock {\em J. Comput. Appl. Math.}, 279:106--122, 2015.

\bibitem{jimenez1999simulation}
JC~Jimenez, I~Shoji, and T~Ozaki.
\newblock Simulation of stochastic differential equations through the local
  linearization method. a comparative study.
\newblock {\em Journal of Statistical Physics}, 94(3-4):587--602, 1999.

\bibitem{Kloedenbook}
P.E. Kloeden and E.~Platen.
\newblock {\em Numerical Solution of Stochastic Differential Equations}.
\newblock Stochastic Modelling and Applied Probability. Springer Berlin
  Heidelberg, 2011.

\bibitem{kloeden2011exponential}
Peter~E Kloeden, Gabriel~J Lord, Andreas Neuenkirch, and Tony Shardlow.
\newblock The exponential integrator scheme for stochastic partial differential
  equations: Pathwise error bounds.
\newblock {\em Journal of Computational and Applied Mathematics},
  235(5):1245--1260, 2011.

\bibitem{YoshiroBurrage}
Yoshio Komori and Kevin Burrage.
\newblock A stochastic exponential {E}uler scheme for simulation of stiff
  biochemical reaction systems.
\newblock {\em BIT}, 54(4):1067--1085, 2014.

\bibitem{gabrielbook}
Gabriel~J Lord, Catherine~E Powell, and Tony Shardlow.
\newblock {\em An Introduction to Computational Stochastic PDEs}.
\newblock Number~50. Cambridge University Press, 2014.

\bibitem{gevrey2004}
Gabriel~J Lord and Jacques Rougemont.
\newblock A numerical scheme for stochastic pdes with {G}evrey regularity.
\newblock {\em IMA journal of numerical analysis}, 24(4):587--604, 2004.

\bibitem{lord2012}
Gabriel~J Lord and Antoine Tambue.
\newblock Stochastic exponential integrators for the finite element
  discretization of {SPDE}s for multiplicative and additive noise.
\newblock {\em IMA Journal of Numerical Analysis}, page drr059, 2012.

\bibitem{MR1475218}
Xuerong Mao.
\newblock {\em Stochastic differential equations and their applications}.
\newblock Horwood Publishing Series in Mathematics \& Applications. Horwood
  Publishing Limited, Chichester, 1997.

\bibitem{Mora}
Carlos~M. Mora.
\newblock Weak exponential schemes for stochastic differential equations with
  additive noise.
\newblock {\em IMA J. Numer. Anal.}, 25(3):486--506, 2005.

\bibitem{Oksendal2003stochastic}
Bernt {\O}ksendal.
\newblock Stochastic differential equations.
\newblock In {\em Stochastic differential equations}, pages 65--84. Springer,
  2003.

\bibitem{geometric}
V~Reshniak, AQM Khaliq, DA~Voss, and G~Zhang.
\newblock Split-step {M}ilstein methods for multi-channel stiff stochastic
  differential systems.
\newblock {\em Applied Numerical Mathematics}, 89:1--23, 2015.

\end{thebibliography}

\end{document}